\documentclass[letterpaper, 10 pt, journal]{IEEEtran}

\usepackage[utf8]{inputenc}    
\usepackage{amssymb}
\usepackage{amsfonts}
\usepackage{amsthm}
\usepackage{amsmath}
{
      \newtheorem{assumption}{Assumption}
  }
\usepackage{graphicx}
\usepackage{mathtools}
\usepackage{xfrac}
\usepackage{nicefrac}
\usepackage{mathrsfs}
\usepackage{breqn}
\usepackage{dsfont}
\usepackage{tabularx}
\usepackage{relsize}
\usepackage{booktabs}
\usepackage{times}
\usepackage{url}
\usepackage{adjustbox}
\usepackage{boldline}
\usepackage{multirow}
\usepackage{algorithm}
\usepackage[T1]{fontenc}
\usepackage{color}
\usepackage{tikz}
\DeclareGraphicsExtensions{.ps,.eps,.pdf}
\newfloat{procedure}{htbp}{loa}
\floatname{procedure}{Procedure}

\IEEEoverridecommandlockouts
\newtheorem{theorem}{Theorem}
\newtheorem{corollary}{Corollary}
\newtheorem{remark}{Remark}

\newtheorem{lemma}{Lemma}

\begin{document} % max 12 pages
\date{}

\title{Set Membership identification of linear systems with guaranteed simulation accuracy}

\author{Marco~Lauricella  and 
	Lorenzo~Fagiano
	\thanks{The authors are with the Dipartimento di Elettronica, Informazione e Bioingegneria, Politecnico di Milano, Piazza Leonardo da Vinci
32, 20133 Milano, Italy. E-mail addresses: \{marco.lauricella $|$ lorenzo.fagiano \}@polimi.it.\newline This is a preprint of a paper published on the IEEE Transactions on Automatic Control, DOI: 10.1109/TAC.2020.2970146.}}

%\markboth{IEEE TRANSACTIONS ON AUTOMATIC CONTROL, VOL. XX, NO. XX, MONTH YEAR}%
%{Lauricella \MakeLowercase{\textit{et al.}: {A Set Membership approach to the identification of linear systems with guaranteed simulation accuracy}}}

\maketitle

\begin{abstract} % max 300 words
The problem of model identification for linear systems is considered, using a finite set of sampled data affected by a bounded measurement noise, with unknown bound. The objective is to identify one-step-ahead models and their accuracy in terms of worst-case simulation error bounds. To do so, the Set Membership identification framework is exploited. Theoretical results are derived, allowing one to estimate the noise bound and system decay rate. Then, these quantities and the data are employed to define the Feasible Parameter Set (FPS), which contains all possible models compatible with the available information. Here, the estimated decay rate is used to refine the standard FPS formulation, by adding constraints that enforce the desired converging behavior of the models' impulse response. Moreover, guaranteed simulation error bounds for an infinite future horizon are derived, improving over recent results pertaining to finite simulation horizon only. These bounds are the basis for a result and method to guarantee asymptotic stability of the identified model. Finally, the desired one-step-ahead model is identified by means of numerical optimization, and the related simulation error bounds are evaluated. Both input-output and state-space model structures are addressed. The approach is showcased on a numerical example and on real-world experimental data of the roll rate dynamics of an autonomous glider.
\end{abstract}

%\begin{IEEEkeywords} System identification, Set Membership identification, Parameter estimation, Simulation error, Error bound \end{IEEEkeywords}

\section{Introduction}
\label{s:intro}
The identification of models with guaranteed simulation accuracy is of great importance in all applications where long range predictions and the related error bounds are used for a robust decision-making task. Examples include resource planning, operations scheduling, and predictive control. In this paper, we address this problem for the case of discrete-time, linear time-invariant systems. 
Our aim is to obtain, from a finite data set, a one-step-ahead model of the system and a measure of its accuracy, in terms of bounds on the simulation error. We want to derive such bounds point-wise in time, for a long, possibly infinite, future simulation horizon, under the action of known future input signals.

The most popular identification procedures are studied in a stochastic framework, see e.g. \cite{ljung1999system}, where theoretical guarantees have been derived assuming that the noise signals are ruled by a probability distribution function. However, many applications feature unknown stochastic properties of the noise, or no sensible statistical hypotheses can be made at all  \cite{fogel1982value}. Motivated by these difficulties, Set Membership identification approaches have been developed under different hypotheses, such as bounded noise and uncertainties, pioneered by \cite{schweppe1968recursive} and \cite{witsenhausen1968sets}. The Set Membership approach provides a way to identify models of systems and to measure their quality without any probabilistic assumptions, referring only to the given data set and noise bounds \cite{kurzhanski1994modeling}, \cite{milanese2013bounding}, \cite{milanese2005model}, \cite{milanese1991optimal}, \cite{walter1992recursive}. In most of the existing works, the noise bound is assumed to be known a priori, which can be a limiting assumption as well. One of the few exceptions is \cite{bai1998convergence}, where the authors propose a way to estimate the noise bound using a probabilistic reasoning.\\
\noindent Another relevant aspect is the purpose of the identification process. Models tuned for multi-step prediction give better performance when used for simulation, e.g. in Model Predictive Control (MPC) schemes, see \cite{farina2011simulation}, \cite{lauri2010pls}. Several approaches address the multi-step-ahead identification problem, see e.g. \cite{lauri2010pls}, \cite{potts2014improving}, \cite{shook1991identification}, \cite{shook1992control}, mainly in a stochastic framework. These approaches do not provide a way to quantify the model quality in terms of bounds on the simulation error, which could be directly exploited in robust decision making.

In this paper, we resort to the Set Membership framework and consider linear systems with bounded noise where, contrary to most existing works, the bound is a-priori unknown. These settings are valid in most real-world applications, where only a rough idea of the noise intensity might be available. We present new theoretical results that allow one to estimate the noise bound from data. A preliminary version of these results has been published in \cite{LF_CDC_18}. Here, we extend the findings to the multiple-input, multiple-output case, and to the case of a predictor structure derived from a state-space representation. Moreover, we introduce a new result to estimate the worst-case simulation error bounds for any simulation horizon, up to infinity. We derive a clear link between the obtained infinite-horizon bound and the estimated noise bounds, model order, system decay rate, and horizon used in the model identification routine. The identification procedure stemming from such theoretical results is composed of four steps: 1) estimation of the noise bound; 2) estimation of the system order; 3) estimation of the impulse responses' decay rates; 4) identification of the model parameters. In this process, the concept of Feasible Parameter Set (FPS) is exploited to define the guaranteed simulation error bounds for a given model, and to constrain the parameters to be identified. We finally prove that the models derived with our procedure are guaranteed to be asymptotically stable, a property that is non-trivial to enforce during the identification phase, see \cite{cerone2011enforcing}. The estimation of the noise bound, of the model order and decay rate, and the analysis of the properties of the finite-horizon and infinite-horizon error bounds, together with the results on the asymptotic stability of the identified models, are the main novelties of our work with respect to the Set Membership literature. %Finally, the desired one-step-ahead model is identified by constrained numerical optimization. Here, we propose two optimality criteria: the minimization of the worst-case error bound over the considered simulation horizon, or the minimization of the simulation error with additional constraints to enforce a desired decay rate of the model response. 
We test the proposed procedure both in a numerical example, where the true quantities are known and the method can be evaluated in full, and in a real-world experimental application, pertaining to the roll rate dynamics of an autonomous glider.

The paper is organized as follows. Section \ref{s:probl_form} contains assumptions and problem formulation. In Section \ref{s:ms_sm_appr} the new theoretical results are presented. Section \ref{s:pred_ident} deals with the identification of the predictor parameters. Section \ref{s:ss_form} extends the obtained results to the state-space model structure with measured state. Section \ref{s:results} presents the numerical and experimental results, and Section \ref{s:conclusions} concludes the paper.

\section{Working assumptions\\ and problem formulation}
\label{s:probl_form}
\subsection{Assumptions on the system, model structure and order}
We consider a discrete time, linear time invariant (LTI) system in the form:
\begin{equation}
\label{eq:sist_desc}
\begin{aligned}
x(k+1)&=Ax(k)+Bu(k) \\
z(k)&=Cx(k),
\end{aligned}
\end{equation}
with state $x(k)\in\mathbb{R}^n$, input $u(k) \in \mathbb{R}^m$ and output $z(k)\in\mathbb{R}^q$. Here $k\in\mathbb{Z}$ denotes the discrete time variable. The output measurement $y(k) \in \mathbb{R}^q$ is affected by an additive noise $d(k)\in\mathbb{R}^q$, leading to:
\begin{equation}
\label{eq:disturbed_output}
y(k)=z(k)+d(k).
\end{equation}
We denote with $z_i(k),\,y_i(k),\,d_i(k)$, the $i$-th component of vectors $z(k),\,y(k),\,d(k)$, respectively, where $i=1,\hdots,q$.
\begin{remark}
	\label{rm:index_range_dropping}
	All of the theoretical developments and practical algorithms have to be applied to each output component individually. Therefore, for the sake of notational simplicity, the notation $i=1,\ldots,q$ will be omitted.
\end{remark}
\begin{assumption}
	\label{as:asympt_stable}
	The system \eqref{eq:sist_desc} is asymptotically stable.
\end{assumption}
\begin{assumption}
\label{as:bounded_dist}
The measurement noise and the system input are bounded. In particular:
\begin{itemize}
\item $|d_i(k)| \leq \bar{d}_{0_i},\; \forall k\in\mathbb{Z}, \; \bar{d}_0 \in \mathbb{R}^{q}.$
\item $u(k)\in\mathbb{U}\subset \mathbb{R}^m, \; \forall k \in \mathbb{Z}, \; \mathbb{U} \; \text{compact}.$
\end{itemize}
\end{assumption}
\begin{assumption}
\label{as:obs_and_reach}
The system \eqref{eq:sist_desc} is fully observable and reachable.
\end{assumption}
\noindent Assumptions \ref{as:asympt_stable} and \ref{as:bounded_dist} are common in system identification problems in real-world applications. Assumption \ref{as:obs_and_reach} is made for simplicity, as it can be relaxed by considering only the observable and controllable sub-space of the system state. Under Assumption \ref{as:obs_and_reach}, for any given $p \in \mathbb{N}$, the output equations can be written in auto-regressive form with exogenous input (ARX):
\begin{equation}
\label{eq:output_gen_pred_form}
z_i(k+p)=\psi_{i_p}(k)^T\theta_{i_p}^0,
\end{equation}
where $^T$ denotes the matrix transpose operation, and the regressor $\psi_{i_p}(k)$ is given by:
\begin{equation}\label{eq:regressor_n}
\begin{array}{rcl}
\psi_{i_p}(k)&=&\left[ Z_{i_n}^T(k) \; U_{p,n}^T(k) \right]^T\in \mathbb{R}^{n+m(n+p-1)}\\
Z_{i_n}(k)&=& \left[ z_i(k) \; z_i(k-1) \; \hdots \; z_i(k-n+1) \right]^T\in \mathbb{R}^{n}\\
U_{p,n}(k)&=&\left[u(k+p-1)^T \; \hdots\right.\\
& & \left. u(k)^T \; \hdots \; u(k-n+1)^T \right]^T\in \mathbb{R}^{m(n+p-1)}.
\end{array}
\end{equation}
In addition, $\theta_{i_p}^0 \in \mathbb{R}^{n+m(n+p-1)}$ is the vector of the true system parameters, which is given by $\theta_{i_p}^0=\left[\theta_{i_{p,z}}^{0^T} \; \theta_{i_{p,u}}^{0^T} \right]^T$, where $\theta_{i_{p,z}}^0$ consists of the parameters related to past values of the output $z_i$, and the entries of $\theta_{i_{p,u}}^0$ are the parameters related to past and future input values. For a discrete time LTI system of the form \eqref{eq:sist_desc}, if all the eigenvalues of $A$ have magnitude strictly smaller than 1 (Assumption \ref{as:asympt_stable}), then, for any initial condition $x_0$ and for any bounded input $u$ such that $\Vert u_i(k)\Vert <M, \, \forall k$, $i=1,\hdots,m$, the system outputs are bounded by
\[ \left\Vert z_i(k) \right\Vert_2 \leq \left\Vert C_i\right\Vert_2\cdot\left\Vert A^k \right\Vert_2\cdot\left\Vert x_0\right\Vert_2 + M \left\Vert C_i\right\Vert_2\cdot \sum_{j=0}^{k-1}\left\Vert A^j\right\Vert_2\cdot \left\Vert B \right\Vert_2, \]
with $i=1,\hdots,q$, $k>0$, and $\Vert A^k \Vert < L \rho^k$, where $0<\rho<1$, $L>0$, see e.g. \cite{zadeh2008linear}. Thus, under Assumption \ref{as:asympt_stable}, the system parameters are bounded by exponentially decaying trends:
\begin{equation}
\label{eq:sists_dec_bound}
\begin{aligned}
&\left\vert \theta_{i_{p,u}}^{0,(l)} \right\vert \leq L_i\rho_i^{\lceil \frac{l}{m} \rceil} \, , \; l=1,\hdots,m(n+p-1) \\
&\left\vert \theta_{i_{p,z}}^{0,(l)} \right\vert \leq L_i\rho_i^{p+l}, \; l=1,\hdots,n
\end{aligned}
\end{equation}
where $^{(l)}$ denotes the $l$-th entry of a vector, $\lceil \; \rceil$ denotes the ceiling function, and $L_i$, $\rho_i$ are scalars that depend on the system matrices in \eqref{eq:sist_desc}.\\
The one-step-ahead dynamics of the system output are then given by \eqref{eq:output_gen_pred_form} with $p=1$. For any $p>1$, the elements of the parameter vector $\theta_{i_p}^0$ are polynomial functions of the entries of $\theta_{i_1}^0$, i.e.:
\begin{equation}
\label{eq:polynomial_functions}
\theta_{i_p}^0=h_{p,n}(\theta_{i_1}^0).
\end{equation}
\noindent The explicit expressions of the polynomial functions $h_{p,n}:\mathbb{R}^{n(m+1)}\rightarrow\mathbb{R}^{n+m(n+p-1)}$ can be readily obtained by recursion of \eqref{eq:output_gen_pred_form} with $p=1$ and are omitted here for simplicity.\\

We consider a model structure given by $q$ one-step-ahead predictors, one for each output signal, written in the ARX form as:
\begin{equation}
\label{eq:1s_pred_classic_form}
\hat{z}_i(k+1)=\varphi_{i_1}(k)^T\theta_{i_1},
\end{equation}
where the regressor $\varphi_{i_p}(k)$ is given by:
\begin{equation}\label{eq:regressor_o}
\begin{array}{rcl}
\varphi_{i_p}(k)&=&\left[ Y_{i_o}^T(k) \; U_{p,o}^T(k) \right]^T\in \mathbb{R}^{o+m(o+p-1)}\\
Y_{i_o}(k)&=& \left[ y_i(k) \; y_i(k-1) \; \hdots \; y_i(k-o+1) \right]^T\in \mathbb{R}^{o}\\
U_{p,o}(k)&=&\left[u(k+p-1)^T \; \hdots\right.\\
& & \left. u(k)^T \; \hdots \; u(k-o+1)^T \right]^T\in \mathbb{R}^{m(o+p-1)}.
\end{array}
\end{equation}
In practice, $\varphi_{i_1}(k)$ is the counterpart of $\psi_{i_1}(k)$ with order $o$ (model order) instead of $n$ (system order), and corrupted by noise (compare \eqref{eq:regressor_n} and \eqref{eq:regressor_o}), while $\theta_{i_1} \in \mathbb{R}^{o(m+1)}$ denotes the vector of model parameters to be identified from data.
\begin{assumption}
	\label{as:model_order}
	The user-selected model order $o$ is such that $o\geq n$.
\end{assumption}
\noindent This assumption is needed to derive part of our theoretical results. In practice, one can initially choose a very large order to make sure that Assumption \ref{as:model_order} is satisfied, and then use our Theorem \ref{th:conv_lambda_diff_d}  and the related Procedure \ref{p:o_est_procedure} (both presented in the next section) to obtain a tighter upper-estimate of $n$.

\subsection{Multi-step predictors and assumption on data}
\label{s:multitep_pred}
In our method, we resort to the concept of multi-step predictors. For a LTI system, the multi-step predictor of the $i$-th system output, pertaining to a given horizon $p>1$, has the following general form:
\begin{equation}
\label{eq:p_pred_classic}
\hat{z}_i(k+p)=\varphi_{i_p}(k)^T\theta_{i_p},
\end{equation}
We refer to $p$ equivalently as the \emph{prediction horizon} or \emph{simulation horizon} in this paper. If the multi-step predictor is obtained by iteration of the one-step-ahead model \eqref{eq:1s_pred_classic_form}, then, similarly to \eqref{eq:polynomial_functions}, the elements of the parameter vector $\theta_{i_p}$ are polynomial functions of the entries of $\theta_{i_1}$, denoted as:
\begin{equation}
\label{eq:polynomial_model}
h_{p,o}:\mathbb{R}^{o(m+1)}\rightarrow\mathbb{R}^{o+m(o+p-1)},\,p\geq 1
\end{equation} 
and obtained by recursion of \eqref{eq:p_pred_classic} with $p=1$.
%\begin{remark}
%	\label{rm:param_padding}
%	In our theoretical analyses we need to evaluate the output given by the true system \eqref{eq:output_gen_pred_form}, defined with a regressor with $n$ past values, in presence of the regressor $\varphi_{i_p}(k)$, featuring $o$ past values. In this case, for the sake of simplicity, recalling Assumption \ref{as:model_order} and with a slight abuse of notation we implicitly consider that the parameter vector $\theta_{i_p}^0$ is appropriately padded with zero entries to match the dimension of $\theta_{i_p}$, thus keeping consistency of the expression . We will use the same convention when evaluating the difference $\theta_{i_p}^0-\theta_{i_1}$, between the parameters of the true system and those of the model.
%\end{remark}

Let us now denote with $\psi_{i_{p_o}}(k)$ the noise-free version of $\varphi_{i_p}(k)$ \eqref{eq:regressor_o}, i.e. using variable $z_i$ instead of $y_i$. Under Assumptions \ref{as:asympt_stable}-\ref{as:bounded_dist}, it follows that:
\begin{equation}\nonumber
\label{eq:psi_set_arx_case}
\psi_{i_{p_o}}(k)\in \Psi_{i_{p_o}}, \, \Psi_{i_{p_o}} \, \text{compact}, \, \forall p \in \mathbb{N}, \, \forall k \in \mathbb{Z}.
\end{equation}
Moreover, also the regressor $\varphi_{i_p}(k)$ belongs to a compact set, indicated as $\Phi_{i_p}$:
\begin{equation}
\label{eq:phi_set_arx_case}
\varphi_{i_p}(k) \in \Phi_{i_p}=\Psi_{i_{p_o}} \oplus \mathbb{D}_{i_p}, \; \forall p \in \mathbb{N}, \; \forall k \in \mathbb{Z},
\end{equation}
where $F\oplus M=\{f+m:f\in F, \, m \in M \}$ is the Minkowski sum of sets $F$, $M$, and $\mathbb{D}_{i_p}\subset \mathbb{R}^{o+m(o+p-1)}$,\small
\begin{equation}
\label{eq:D_arx_def}
\mathbb{D}_{i_p} \doteq\left\{ \left[ d_i^{(1)}, \hdots, d_i^{(o)}, 0, \hdots, 0 \right]^T: \left\vert d_i^{(l)} \right\vert \leq \bar{d}_{0_i},\,l=1,\ldots,o \right\}.
\end{equation}\normalsize
Namely, $\mathbb{D}_{i_p}$ is the set of all possible noise realizations that can affect the  system output values in $\varphi_{i_p}$. 

We assume that a finite number of measured pairs $(\tilde{y}(k),\tilde{u}(k))$ is available for the model identification task, where $\tilde{\cdot}$ is used to denote a sample of a given variable. For each simulation horizon $p$, these data form the following set of sampled regressors and corresponding output values:
\begin{equation}
\label{eq:samp_dataset_arx}\nonumber
\tilde{\mathscr{V}}_{i_p}^N \doteq \left\{ \tilde{v}_{i_p}(k)=\begin{bmatrix} \tilde{\varphi}_{i_p}(k) \\ \tilde{y}_{i_p}(k) \end{bmatrix}, \, k=1,\hdots,N \right\}, 
\end{equation}
where $\tilde{\mathscr{V}}_{i_p}^N \subset \mathbb{R}^{1+o+m(o+p-1)}$ and $\tilde{y}_{i_p}(k) \doteq \tilde{y}_i(k+p)$. Here, for simplicity and without loss of generality, we consider that the number of sampled regressors $N$ is the same for any considered value of $p$. The set $\tilde{\mathscr{V}}_{i_p}^N$ can be seen as a countable, sampled version of its continuous counterpart, $\mathscr{V}_{i_p}$:
\begin{equation}
\label{eq:cont_dataset_arx}\nonumber
\mathscr{V}_{i_p} \doteq \left\{ v_{i_p}=\begin{bmatrix} \varphi_{i_p} \\ y_{i_p} \end{bmatrix}, \, y_{i_p} \in Y_{i_p}(\varphi_{i_p}), \, \forall \varphi_{i_p} \in \Phi_{i_p} \right\},
\end{equation}
where $\mathscr{V}_{i_p} \subset \mathbb{R}^{1+o+m(o+p-1)}$, and $Y_{i_p}(\varphi_{i_p}) \subset \mathbb{R}$ represents the compact set of all the possible output values corresponding to each regressor $\varphi_{i_p} \in \Phi_{i_p}$ and to every possible noise realization $d_i:|d_i|\leq \bar{d}_{0_i}$.

Let us define the distance between $\tilde{\mathscr{V}}_{i_p}^N$ and $\mathscr{V}_{i_p}$ as:
\[
d_2 \left( \mathscr{V}_{i_p},\tilde{\mathscr{V}}_{i_p}^N \right)\doteq \underset{v_1 \in \mathscr{V}_{i_p}^{\phantom{0}}}{\textrm{max}} \underset{v_2 \in \tilde{\mathscr{V}}_{i_p}^{N}}{\textrm{min}} \left\Vert v_2-v_1 \right\Vert_2
\]
We consider the following assumption on the data set:
\begin{assumption}
\label{as:info_data}
For any $\beta>0$, there exists a value of $N<\infty$ such that
$d_2 \left( \mathscr{V}_{i_p},\tilde{\mathscr{V}}_{i_p}^N \right) \leq \beta$.
\end{assumption}
\noindent Assumption \ref{as:info_data} pertains to the informative content of the sampled data set. It means that, by adding more points to $\tilde{\mathscr{V}}_{i_p}^N$, the set of all the system trajectories of interest is densely covered. This can be seen as a persistence of excitation condition combined with a bound-exploring property of the noise signal $d$.

\subsection{Problem formulation}
We are now in position to formalize the problem addressed in this paper.
\\$\,$\\
\fbox{\parbox{0.98\columnwidth}{
		\textbf{Problem 1}. Under Assumptions \ref{as:asympt_stable}-\ref{as:info_data}, use the available data sets $\tilde{\mathscr{V}}_{i_p}^N$ to:
%		\begin{enumerate}
%			\item Estimate the noise bounds $\bar{d}_{0_i}$;
%			\item Select the model order $o\approx n$;
%			\item Estimate the parameters $L_i$, $\rho_i$ defining the system's decaying trend \eqref{eq:sists_dec_bound};
%			\item For a given value of model parameters $\theta_{i_1}$ \eqref{eq:1s_pred_classic_form}, estimate worst-case bounds on the simulation error $z(k+p)-\hat{z}_i(k+p)$ for $p\in[0,\infty)$;
%			\item Identify the model parameters $\theta_{i_1}$ exploiting the knowledge of the estimated noise bound, model order, decay rate and simulation error bounds.		
%		\end{enumerate}}}
\begin{enumerate}
			\item Estimate the noise bounds $\bar{d}_{0_i}$;
			\item Select the model order $o\approx n$;
			\item Estimate the parameters $L_i$, $\rho_i$ defining the system's decaying trend \eqref{eq:sists_dec_bound};
			\item Identify the model parameters $\theta_{i_1}$ exploiting the knowledge of the estimated quantities;%noise bound, model order and decay rate;
			\item For the model parameters $\theta_{i_1}$ obtained from the previous step, estimate worst-case bounds on the simulation error $z(k+p)-\hat{z}_i(k+p)$ for $p\in[0,\infty)$.
		\end{enumerate}}}
%$\,$\\
%In Sections \ref{s:ms_sm_appr}-\ref{s:pred_ident} we provide results and procedures to accomplish these tasks, while in Section \ref{s:ss_form} we provide an extension to models of the form \eqref{eq:sist_desc}, for the case of measured state and known system order.

\section{Estimation of the noise bound, model order, decay trend, and simulation error bounds}
\label{s:ms_sm_appr}
The key to address points 1)-4) of \textbf{Problem 1} is the analysis of the multi-step predictors of the form \eqref{eq:p_pred_classic}. At first, we will consider the multi-step predictor for each simulation horizon $p$  as an independent function, neglecting the fact that the true system \eqref{eq:output_gen_pred_form} (and the wanted model \eqref{eq:1s_pred_classic_form}) define implicitly multi-step predictors, whose parameters are linked by polynomial functions $h_{p,n}(\cdot)$ \eqref{eq:polynomial_functions} (and $h_{p,o}(\cdot)$ \eqref{eq:polynomial_model}). We will introduce such links later on, as constraints in the identification procedures of Section \ref{s:pred_ident}.\\
The starting base for our new results are the findings described in \cite{TFFS}, briefly recalled next.
\subsection{Preliminary results}
\label{ss:preliminary_res}
Under Assumptions \ref{as:asympt_stable}-\ref{as:bounded_dist}, the error between the true $p$-steps-ahead system output and its prediction \eqref{eq:p_pred_classic} is bounded for any finite $p$:
\begin{equation}\nonumber
\label{eq:global_eps_def}
\left\vert y_i(k+p)-\varphi_{i_p}^T \theta_{i_p} \right\vert \leq \bar{\varepsilon}_{i_p}(\theta_{i_p})+\bar{d}_i,
\end{equation}
where $\bar{d}_i\geq 0$ denotes an estimate of the true noise bound $\bar{d}_{0_i}$, and $\bar{\varepsilon}_{i_p}(\theta_{i_p})$ represents the global error bound related to given multi-step model parameters $\theta_{i_p}$, i.e. it holds for all the possible values of $\varphi_{i_p}$ in $\Phi_{i_p}$. 
\noindent Theoretically, the global error bound $\bar{\varepsilon}_{i_p}(\theta_{i_p})$ is the solution to the following optimization problem:
\begin{equation}
\label{eq:global_eps_calc}
\begin{array}{c}
\bar{\varepsilon}_{i_p}(\theta_{i_p}) = \min\limits_{\varepsilon\in \mathbb{R}^+} \varepsilon \\
\text{subject to} \\
\left\vert y_{i_p}-\varphi_{i_p}^T \theta_{i_p} \right\vert \leq \varepsilon +\bar{d}_i, \; \forall \left(\varphi_{i_p},y_{i_p} \right): \begin{bmatrix} \varphi_{i_p} \\ y_{i_p} \end{bmatrix} \in \mathscr{V}_{i_p}.
\end{array}
\end{equation}
Moreover, among all possible parameter values, one is interested in those that minimize the corresponding global error bound: 
\begin{equation}
\label{eq:opt_multistep}
\bar{\varepsilon}_{i_p}^0= \min\limits_{\theta_{i_p}\in \Omega_p} \bar{\varepsilon}_{i_p}(\theta_{i_p}),
\end{equation}
where the set $\Omega_p$ is a compact approximation of $\mathbb{R}^{o+m(o+p-1)}$ (e.g. an hypercube defined by $\|\theta_{i_p}\|_\infty\leq10^{100}$) introduced to technically replace $\inf$  and $\sup$ operators with $\min$ and $\max$, respectively.\\
\noindent Problems \eqref{eq:global_eps_calc}-\eqref{eq:opt_multistep} are intractable. Using the available finite set of data points, one can however compute an estimate $\underline{\lambda}_{i_p}\approx\bar{\varepsilon}_{i_p}^0$ solving the following Linear Program (LP):
\begin{equation}
\label{eq:lambda_p_calc}
\begin{array}{c}
\underline{\lambda}_{i_p} = \min
\limits_{\theta_{i_p}\in\Omega_p,\,\lambda\in\mathbb{R}^+} 
\lambda \\
\text{subject to} \\
\left\vert \tilde{y}_{i_p}-\tilde{\varphi}_{i_p}^T \theta_{i_p} \right\vert \leq \lambda +\bar{d}_i, \; \forall \left(\tilde{\varphi}_{i_p},\tilde{y}_{i_p} \right) : \begin{bmatrix} \tilde{\varphi}_{i_p} \\ \tilde{y}_{i_p} \end{bmatrix} \in \tilde{\mathscr{V}}_{i_p}^N.
\end{array}
\end{equation}
Under Assumptions \ref{as:bounded_dist}-\ref{as:info_data}, the following properties hold (see \cite{TFFS} for the derivation):
\begin{subequations}\label{eq:properties_preliminary}
	\begin{equation}
	\underline{\lambda}_{i_p} \leq \bar{\varepsilon}_{i_p}^0\label{eq:properties_preliminary_a}
	\end{equation}
\begin{equation}
\forall \eta \in (0,\bar{\varepsilon}_{i_p}^0], \; \exists N < \infty: \; \underline{\lambda}_{i_p} \geq \bar{\varepsilon}_{i_p}^0-\eta\label{eq:properties_preliminary_b}
\end{equation}
\end{subequations}
i.e. the estimated bound $\underline{\lambda}_{i_p}$ converges to $\bar{\varepsilon}_{i_p}^0$ from below. 

\subsection{Theoretical properties of the multi-step error bound}
\label{ss:prop_lambda}
In \cite{TFFS}, the results \eqref{eq:properties_preliminary} are exploited to build a FPS for any finite value of $p$ and to estimate the worst-case error of a given multi-step predictor, again for a finite simulation horizon. However, no result and/or systematic procedure to fulfill the assumptions on the noise bound (supposed to be known in \cite{TFFS}) and the model order were provided. These aspects limit the applicability of the approach, since in practice the true values of $\bar{d}_{0_i}$ and $n$ are often unknown and one has to resort to heuristics to choose $\bar{d}_i$ and $o$. We now introduce two new results that solve this issue, allowing one to derive a convergent estimate $\bar{d}_{i}\approx\bar{d}_{0_i}$, as well as estimates of the system order and, additionally, of the impulse response decay trend. The main conceptual innovation with respect to the preliminary results of \cite{TFFS} is to analyze not only each value of $\underline{\lambda}_{i_p}$ separately, but also the course of this quantity as a function of the  horizon $p$.

%Let us define $\lambda_{i_p}$ as:
%\begin{equation}
%\label{eq:lp_equality}
%\lambda_{i_p} \doteq \min\limits_{\theta_{i_p} \in \Omega_p} \max_{\left[\begin{smallmatrix} \varphi_{i_p} \\ y_{i_p} \end{smallmatrix}\right] \in \mathscr{V}_{i_p}} \left( \left\vert y_{i_p} - \varphi_{i_p}^T \theta_{i_p} \right\vert -\bar{d}_i \right).
%\end{equation}
%$\lambda_{i_p}$ can be seen as the theoretical upper bound of $\underline{\lambda}_{i_p}$ \eqref{eq:lambda_p_calc} with infinitely many data points, or alternatively as the global error bound $\bar{\varepsilon}_{i_p}^0$ \eqref{eq:opt_multistep} but using the estimated noise bound $\bar{d}_i$ instead of the true one, $\bar{d}_{0_i}$.
\begin{theorem}
\label{th:conv_lambda_diff_d}
If Assumptions \ref{as:bounded_dist}-\ref{as:info_data} hold, then, for any arbitrarily small $\eta>0$, $\exists N < \infty$ such that
\begin{equation}\label{eq:thm1_convergence}
\left( \bar{d}_{0_i} - \bar{d}_i \right)-\eta \leq \lim_{p \to \infty}{\underline{\lambda}_{i_p}} \leq \left( \bar{d}_{0_i} - \bar{d}_i \right).
\end{equation}
%\begin{equation}\label{eq:thm1_convergence}
%\underline{\lambda}_{i_p} \xrightarrow{p \to \infty} \left( \bar{d}_{0_i} - \bar{d}_i \right)-\eta
%\end{equation}
\end{theorem}
\begin{proof}
See the Appendix.
\end{proof}

\begin{corollary}
\label{c:lambda_rate_conv}
If Assumptions \ref{as:bounded_dist}-\ref{as:info_data} hold, and if the estimated noise bound is correctly chosen as $\bar{d}_i=\bar{d}_{0_i}$, then, for any arbitrarily small $\eta>0$, $\exists N < \infty$ such that
\begin{equation}\label{eq:corollary1_bound}\nonumber
\underline{\lambda}_{i_p} = \left\Vert \theta^0_{i_p} \right\Vert_1 \bar{d}_{0_i}-\eta \leq n \bar{d}_{0_i} L_i \rho_i^{p+1}.
\end{equation}
\end{corollary}
\begin{proof}
See the Appendix.
\end{proof}
\begin{remark}
\label{rm:consequence_of_th_lambda}
Theorem \ref{th:conv_lambda_diff_d} and Corollary \ref{c:lambda_rate_conv} imply three consequences useful to estimate the noise bound and system order:
\begin{enumerate}
	\item With a large enough data set, the estimated bound $\underline{\lambda}_{i_p}$ \eqref{eq:lambda_p_calc} converges, as $p$ increases, to the difference between the true noise bound, $\bar{d}_{0_i}$, and the estimated one, $\bar{d}_i$. We will use this result to estimate $\bar{d}_{0_i}$;
	\item When $\bar{d}_i=\bar{d}_{0_i}$ and $o<n$ (i.e. Assumption \ref{as:model_order} is not met), then $\underline{\lambda}_{i_p}$ converges (besides a quantity $\eta$ that can be made arbitrarily small with a larger data set) to a non-zero value as $p \to \infty$, due to model order mismatch (see the proof of Theorem \ref{th:conv_lambda_diff_d} for more details). We will exploit this property to estimate the system order;
	\item Assuming the noise bound is chosen as $\bar{d}_i\simeq\bar{d}_{0_i}$, then the estimated bound $\underline{\lambda}_{i_p}$ converges to zero as $p\to \infty$ with the same decay trend as that of the true system parameters, dictated by the system dominant eigenvalues. We will exploit this property to estimate the system decay rate.
\end{enumerate}
\end{remark}
\subsection{Estimation of noise bound, system order and decay trend}
\label{ss:estim_d_o_rho}
We propose three procedures to estimate the noise bound, system order and the decay trend, respectively. This information will be used in Section \ref{ss:FPS_and_tau_def} to define the FPS for any finite $p$ and the guaranteed simulation error bound of any predictor up to a finite $p$.\\
We start by estimating $\bar{d}_{0_i}$ resorting to Theorem \ref{th:conv_lambda_diff_d} (see also point 1) of Remark \ref{rm:consequence_of_th_lambda}):
\begin{procedure}
\caption{Estimation of $\bar{d}_0$}
\label{p:d_bar_est_procedure}
Choose a large value as initial guess of $o$. Then, for all $i=1,\ldots,q$, carry out the following steps:
\begin{enumerate}
\item Initialize $\bar{d}_i$ with a value small enough to ensure $\bar{d}_i<\bar{d}_{0_i}$ (e.g. $\bar{d}_i\simeq0$);
\item Compute $\underline{\lambda}_{i_p}$ \eqref{eq:lambda_p_calc} for increasing $p$ values, until it converges to a constant quantity $e_{d_i}\simeq(\bar{d}_{0_i}-\bar{d}_i)$ as $p \to \infty$; \item Correct the initial guess of $\bar{d}_i$ by adding $e_{d_i}$;
\item Verify that $\underline{\lambda}_{i_p} \xrightarrow{p \to \infty}0$ with the new value of $\bar{d}_i$;
\end{enumerate}
Take the resulting vector $\bar{d}=[\bar{d}_1,\hdots,\bar{d}_q]^T$ as estimate of the true one, $\bar{d}_0$.
\end{procedure} 
\\
\noindent This addresses point 1) of \textbf{Problem 1}. After completing Procedure \ref{p:d_bar_est_procedure}, we can compute a finite simulation horizon value  $\bar{p}_i$ such that:
\begin{equation}
\label{eq:bar_p}
\begin{array}{c}
\bar{p}_i=\min\limits_{\bar{p}\in\mathbb{N}}\bar{p}\\
\text{subject to}\\
\underline{\lambda}_{i_p}<\delta,\, \forall p \geq \bar{p}
\end{array}
\end{equation}
where $\delta\approx 0$ is a suitable tolerance, e.g. $10^{-8}$, to account for the asymptotic behavior of $\underline{\lambda}_{i_p}$ (see Theorem \ref{th:conv_lambda_diff_d}). This tolerance can be used to check the convergence of $\underline{\lambda}_{i_p}$ in step 5 of Procedure \ref{p:d_bar_est_procedure}, i.e. to verify that $\exists \tilde{p}: \underline{\lambda}_{i_p}<\delta \; \forall p>\tilde{p}$ \\
Exploiting the values of $\bar{p}_i$, we can then estimate the system order resorting to the observation of point 2) of Remark \ref{rm:consequence_of_th_lambda}:
\begin{procedure}
\caption{Estimation of $n$}
\label{p:o_est_procedure}
\begin{enumerate}
\item Set $\bar{d}_i$ to the values resulting from Procedure \ref{p:d_bar_est_procedure}, and compute $\bar{p}_i$ as in \eqref{eq:bar_p};
\item Set a large starting value of $o$;
\item Gradually decrease $o$, recalculating all the $\underline{\lambda}_{i_p}$, until a value of $o$ is found, such that 
$\exists p>\bar{p}_i : \underline{\lambda}_{i_p}>\delta$, with $\delta$ used in \eqref{eq:bar_p}. Denote as $\underline{o}$ such a value.
\item Set the model order as $o=\underline{o}+1$.
\end{enumerate}
\end{procedure}
\\ \noindent This addresses point 2) of \textbf{Problem 1}. At the end of Procedure \ref{p:o_est_procedure}, one shall choose the largest value of $o$ among all $i=1,\ldots,q$.
Finally, we estimate the system decay trend from that of $\underline{\lambda}_{i_p}$, exploiting observation 3) of Remark \ref{rm:consequence_of_th_lambda}, thus addressing also point 3) of Problem 1.
\begin{procedure}
	\caption{Estimation of $L_i$ and $\rho_i$}
	\label{p:decay_est_procedure}
	\begin{enumerate}
		\item Take $\bar{d}_i$, $\bar{p}_i$, and $o$ resulting from Procedures \ref{p:d_bar_est_procedure}-\ref{p:o_est_procedure};
		\item Compute two scalars, $L_i',\,\hat{\rho}_i$ as:
		\begin{equation}
\label{eq:est_L_rho_optprob}
\begin{aligned}
\left[L_i', \hat{\rho}_i \right] = &\text{arg} \min_{L_i',\hat{\rho}_i} \left\Vert \boldsymbol{f}_{i_{\lambda}} - \boldsymbol{g}_{i_{L\rho}} \right\Vert_2^2 \\
&\text{subject to} \\
&\boldsymbol{g}_{i_{L\rho}} \succeq \boldsymbol{f}_{i_{\lambda}} \\
&L_i'>0 \\
&0< \hat{\rho}_i <1
\end{aligned}
\end{equation}
where $\boldsymbol{f}_{i_{\lambda}} \doteq [\underline{\lambda}_{i_1} \; \cdots \; \underline{\lambda}_{i_{\bar{p}_i}}]^T$, $\boldsymbol{g}_{i_{L\rho}}\doteq [g_{i_{L\rho}}(1) \; \cdots \; g_{i_{L\rho}}(\bar{p}_i) ]^T$, $g_{i_{L\rho}}(p)=L_i' \hat{\rho}_i^{p+1}$, and $\succeq$ denotes element-wise inequalities.
\item Compute $\hat{L}_i$ as (from Corollary \ref{c:lambda_rate_conv}):
\begin{equation}
\label{eq:L_eff_arx}
\hat{L}_i=\dfrac{L_i'}{o\bar{d}_i},
\end{equation}
\item Set $\hat{L}_i,\,\hat{\rho}_i$ as estimates of $L_i$ and $\rho_i$, respectively.
\end{enumerate}
\end{procedure}
\\ \noindent Problem \eqref{eq:est_L_rho_optprob} is always feasible, since one can always choose large-enough values of $L_i'$ to satisfy its constraints. Moreover, the cost function results to be convex inside the feasible set, as it can be shown by computing its curvature and checking that it is positive for feasible $(L_i',\hat{\rho}_i)$ pairs. %Thus, problem \eqref{eq:est_L_rho_optprob} can be readily solved for a global minimum using standard NLP solvers.
%\begin{remark}
%Thanks to the results demonstrated in \cite{TFFS}, we have guarantees that the estimated error bound $\underline{\lambda}_{i_p}$ converges to the true bound $\lambda_{i_p}$ as the number of data points $N$ increases, meaning that also $\underline{\lambda}_{i_p}$ undergoes by the properties resulting from Theorem \ref{th:conv_lambda_diff_d} and Corollary \ref{c:lambda_rate_conv}. This provides us with the possibility to estimate the noise bound amplitude, the system order and the decay rate also with a finite data set. \qed
%\end{remark}
%
\subsection{Feasible Parameter Sets and finite-horizon simulation error bound}
\label{ss:FPS_and_tau_def}

The quantities estimated so far are instrumental to build the Feasible Parameter Set (FPS) for any finite simulation horizon $p$. Namely, such sets contain all possible multi-step predictor parameters $\theta_{i_p}$ that are consistent with the available data set, up to the tolerance given by the global error bound $\bar{\varepsilon}_{i_p}^0$ and noise bound $\bar{d}_{0_i}$, and the other available information on the system at hand. Since the  computed bound $\underline{\lambda}_{i_p}$ is lower than $\bar{\varepsilon}_{i_p}^0$, due to the use of a finite data set (property \eqref{eq:properties_preliminary_a}), it is customary to employ a scaling factor $\alpha>1$ to estimate the global error bound:
\begin{equation}
\label{eq:eps_hat_def}
\hat{\bar{\varepsilon}}_{i_p} = \alpha \underline{\lambda}_{i_p}, \; \alpha>1.
\end{equation}
We can now define, for the $p$-steps-ahead predictor of the $i$-th system output, the set $\Theta_{i_p}$ of parameter values that are consistent with the measured data, and with the estimated noise bound and global error bound. Several works in the Set Membership literature prefer to lower the computational effort resorting to outer approximation of the FPS, e.g. via intervals \cite{sun2003recursive}, ellipsoid \cite{bertsekas1971recursive,durieu2001multi,filippova1996ellipsoidal}, parallelotopes \cite{chisci1996recursive,vicino1996sequential}, zonotopes \cite{alamo2005guaranteed,bravo2006bounded,combastel2003state,wang2018zonotope}, or constrained zonotopes \cite{scott2016constrained}. \cite{walter1989exact} proposes a recursive exact polytopic representation, able to cope also with time-varying systems. Here, we decided to adopt an exact description of the FPS by defining it as a polytope using an inequality description (H-representation):
\begin{equation}
\label{eq:FPS_gen_def}\nonumber
\begin{aligned}
\Theta_{i_p} \doteq \bigg\{ \theta_{i_p} &: \left\vert \tilde{y}_{i_p} - \tilde{\varphi}_{i_p}^T \theta_{i_p} \right\vert \leq \hat{\bar{\varepsilon}}_{i_p}+\bar{d}_i, \\
&\quad \forall \left(\tilde{\varphi}_{i_p},\tilde{y}_{i_p} \right) : \begin{bmatrix} \tilde{\varphi}_{i_p} \\ \tilde{y}_{i_p} \end{bmatrix} \in \tilde{\mathscr{V}}_{i_p}^N \bigg\}.
\end{aligned}
\end{equation}
The set $\Theta_{i_p}$, if bounded, is a polytope with at most $2N$ facets. If $\Theta_{i_p}$ is unbounded, then this indicates that the data collected from the system are not informative enough, and new data should be acquired. In \cite{TFFS}, the set $\Theta_{i_p}$ was taken as FPS for the predictors pertaining to the horizon $p$. Here, we provide a further refinement by adding the constraints on the estimated decay trend obtained in Section \ref{ss:estim_d_o_rho}. Let us define the polytope:
\begin{equation}
\label{eq:gamma_set_def_arx}\nonumber
\begin{aligned}
\Gamma_{i_p}\doteq \bigg\{ \theta_{i_p}: &\left\vert \theta_{i_{p,z}}^{(l)} \right\vert \leq \hat{L}_i \hat{\rho}_i^{p+l}, \; l\in [1,o], \\
&\wedge \left\vert \theta_{i_{p,u}}^{(l)} \right\vert \leq \hat{L}_i \hat{\rho}_i^{\lceil \frac{l}{m} \rceil}, \; l\in [1,m(o+p-1)] \bigg\}.
\end{aligned}
\end{equation}
Then, we define the Feasible Parameter Sets as:
\begin{equation}
\label{eq:FPS_with_decay}
\Theta_{i_p}^{L\rho}\doteq\Theta_{i_p} \cap \Gamma_{i_p}.
\end{equation}

Note that this new FPS is always compact, since $\Gamma_{i_p}$ is. The FPS is used to derive the worst-case simulation error bound obtained by a given predictor with parameters $\theta_{i_p}$:
\begin{equation}
\label{eq:tau_global_def}
\tau_{i_p}(\theta_{i_p})= \max_{\varphi_{i_p} \in \Phi_{i_p}^{\phantom{0}}} \max_{\theta \in \Theta_{i_p}^{L\rho}} \left\vert \varphi_{i_p}^T \left(\theta - \theta_{i_p} \right) \right\vert + \hat{\bar{\varepsilon}}_{i_p}.
\end{equation}
Namely, this bound is the  worst-case absolute difference between the output $\hat{z}_i(k+p)=\varphi_{i_p}^T(k) \theta_{i_p}$, predicted using the parameters $\theta_{i_p}$, and the one predicted by any other parameter vector in the FPS, plus the worst-case prediction error $\hat{\bar{\varepsilon}}_{i_p}$ related to all $\theta_{i_p}\in\Theta_{i_p}^{L\rho}$. In a way similar to $\bar{\varepsilon}_{i_p}^0$, it is not possible to exactly compute the bound  \eqref{eq:tau_global_def} using a finite data set. Thus, we introduce an estimate  $\hat{\tau}_{i_p}(\theta_{i_p})$, which, under Assumption \ref{as:info_data}, converges to $\tau_{i_p}(\theta_{i_p})$ from below as the number of data points increases, see \cite{TFFS}. Such an estimate is then inflated by a scalar $\gamma>1$ to account for the uncertainty due to the usage of a finite data set:
\begin{equation}
\label{eq:tau_hat_global_def}
\hat{\tau}_{i_p}(\theta_{i_p})=\gamma \left( \max_{\tilde{\varphi}_{i_p}\in \tilde{\mathscr{V}}_{i_p}^N} \max_{\theta \in \Theta_{i_p}^{L\rho}} \left\vert \tilde{\varphi}_{i_p}^T \left(\theta - \theta_{i_p} \right) \right\vert \right) + \hat{\bar{\varepsilon}}_{i_p}, \; \gamma>1.
\end{equation}
The estimation of the bound defined by \eqref{eq:tau_hat_global_def}, corresponding to point 5) of \textbf{Problem 1}, will then be performed on the models identified using the approaches proposed in Section \ref{s:pred_ident}. Note that \eqref{eq:tau_hat_global_def} can be recast as an LP, after the preliminary solution of $2N$ LPs which can be parallelized. If the estimated error bounds $\hat{\bar{\varepsilon}}_{i_p}$ and $\hat{\tau}_{i_p}(\theta_{i_p})$ are larger than the corresponding theoretical values $\bar{\varepsilon}_{i_p}^0$ and $\tau_{i_p}(\theta_{i_p})$, respectively, and the estimated decay rate parameters are such that $\hat{\rho} \in [\rho,1)$ and $\hat{L}_i \geq L_i$, then it is easy to show that the multi-step predictor $\theta_{i_p}^0$, obtained from the true system and possibly appropriately padded with zero entries if $o>n$, belongs to the FPS $\Theta_{i_p}^{L\rho}$. In this case, by construction, the bound $\hat{\tau}_{i_p}(\theta_{i_p})$ is such that:
\begin{subequations}\label{eq:z_under_bound_tau}
\begin{eqnarray}
|z_i(k+p)-\hat{z}_i(k+p) | &\leq& \hat{\tau}_{i_p}(\theta_{i_p})\label{eq:sim_bound_z}\\
|y_i(k+p)-\hat{z}_i(k+p) | &\leq& \hat{\tau}_{i_p}(\theta_{i_p})+\bar{d}_i\label{eq:sim_bound_y}
\end{eqnarray}
\end{subequations}
i.e. it is the desired simulation error bound for the considered finite horizon $p$.

The parameters $\alpha$ in \eqref{eq:eps_hat_def} and $\gamma$ in \eqref{eq:tau_hat_global_def} essentially express how much we are confident in the informative content of the data set. A ``large'' value of $\alpha$ might produce an overly conservative error bound $\hat{\bar{\varepsilon}}_{i_p}$ and, consequently, larger FPSs, while a choice of $\alpha$ close to 1  might produce an error bound that could be invalidated by future data, if the available data set has a poor informative content. Similarly, a ``large'' $\gamma$ might give a conservative bound $\hat{\tau}_{i_p}$.\\
\begin{remark}\label{rm:bounds_conservativeness}
In a real application, one will never know whether the scaling factors $\alpha,\gamma$ are too conservative. Conversely, it is easy to understand when these factors are too small, by checking whether the FPS is empty for any $p$. If this happens, for example if one chooses a too small $\alpha$ value, then the prior assumptions and/or estimated bounds are invalidated by data. Thus, verifying that all the FPSs are non-empty (which is an easy task since they are all polytopes) is a way to check the informative content of our data set and the choice of parameter $\alpha$.  This check can be carried out using new data collected in a validation experiment, or in real-time if the FPSs and system model are to be updated on-line. A similar reasoning applies to the bound $\hat{\tau}_{i_p}(\theta_{i_p})$ and scalar $\gamma$: conservativeness can be evaluated by checking the bound against new measured data and evaluating whether the simulation error magnitude ever violates  $\hat{\tau}_{i_p}(\theta_{i_p})$ by more than $\bar{d}_i$ (see \eqref{eq:sim_bound_y}).
\end{remark}

\subsection{Infinite-horizon simulation error bound}
\label{s:tau_inf_intro}
The error bound \eqref{eq:tau_hat_global_def} requires the computation of the FPSs for each horizon $p$  of interest, potentially up to a very large value. Since each FPS is a polytope whose complexity generally grows with the number of available data, the computation of a large number of bounds $\hat{\tau}_{i_p}$ can become impractical. To solve this problem, in this section we present new results that allow one to estimate the simulation error bound for any future horizon, beyond a (sufficiently large) finite value $\bar{p}$. In particular, we propose an iterative expression to compute the simulation error bound for $p>\bar{p}$, based on the previous computation of the bounds $\hat{\tau}_{i_p}$ for $p=1,\ldots,\bar{p}$. Furthermore, we provide results indicating how the value of $\bar{p}$ should be chosen in order to keep the computational effort at a minimum, and obtain a bound which is non-divergent with $p$ and not excessively conservative. Before proceeding further, the following remark is in order.
\begin{remark}
\label{rm:model_inf_tau}
The results presented in the remainder of this section are derived considering model parameters that satisfy the conditions $h_{p,o}(\theta_{i_1}) \in \Gamma_{i_p}, \; \forall p \in [2,\bar{p}]$. Later on, in Section \ref{s:pred_ident}, we will include explicitly such conditions in the identification procedure, so that the computed models will always enjoy this property. This establishes a connection between the derived theoretical results and the proposed computational methods to identify a model.
\end{remark}

Given the multi-step predictors described by \eqref{eq:p_pred_classic}, and having computed the error bounds defined by \eqref{eq:tau_hat_global_def} up to $\bar{p}$, if $h_{p,o}(\theta_{i_1}) \in \Gamma_{i_p}, \; \forall p \in [2,\bar{p}]$, the simulation error at horizon $\bar{p}+j$, $j>1$, is such that:
\begin{equation}
\label{eq:err_bound_arx_pgrande}
\begin{array}{l}
\left\vert z_i(k+\bar{p}+j) - \hat{z}_i(k+\bar{p}+j) \right\vert \leq \\
\hat{\tau}_{i_{\bar{p}}}(\theta_{i_{\bar{p}}})+ \sum\limits_{m=1}^{\min\{j,o\}} \left( \hat{\tau}_{i_{j-m+1}}(\theta_{i_{j-m+1}}) +\bar{d}_i \right)\hat{L}_i\hat{\rho}_i^{\bar{p}+m}.
\end{array}
\end{equation}
See the Appendix for a derivation. Then, considering that
$\hat{L}_i\hat{\rho}_i^{\bar{p}+1}>\hat{L}_i\hat{\rho}_i^{\bar{p}+2}>\hdots>\hat{L}_i\hat{\rho}_i^{\bar{p}+o}$,
and that
$
\sum\limits_{m=1}^{\min\{j,o\}}\hat{L}_i\hat{\rho}_i^{\bar{p}+m} \leq o\hat{L}_i\hat{\rho}_i^{\bar{p}+1},
$
we can derive an over-estimate of the simulation error bound $\hat{\tau}_{i_{\bar{p}+j}}$ as:
\begin{equation}
\label{eq:tau_arx_pgrande_upbound}
\begin{array}{l}
\left\vert z_i(k+\bar{p}+j) - \hat{z}_i(k+\bar{p}+j) \right\vert \leq
\hat{\tau}_{i_{\bar{p}+j}}(\theta_{i_{\bar{p}+j}})\\ \leq \hat{\tau}_{i_{\bar{p}}}(\theta_{i_{\bar{p}}})+ \left( \hat{\tau}_{i_{max_{\{j,o\}}}}+\bar{d}_i \right)o\hat{L}_i\hat{\rho}_i^{\bar{p}+1},
\end{array}
\end{equation}
where $\hat{\tau}_{i_{max_{\{j,o\}}}}=\max\{\hat{\tau}_{i_{j-o}}(\theta_{i_{j-o}}), \hdots, \hat{\tau}_{i_j}(\theta_{i_j}) \}$. \\
Note that, if $j\geq \bar{p}+1$, the term $\hat{\tau}_{i_j}(\theta_{i_j})$ is not computed using \eqref{eq:tau_hat_global_def}, but resorting to \eqref{eq:tau_arx_pgrande_upbound}. For example, when $j \in (\bar{p}, 2\bar{p}]$, the simulation error bound $\hat{\tau}_{i_{\bar{p}+j}}(\theta_{i_{\bar{p}+j}})$ is bounded as:
\begin{equation}
\label{eq:tau_arx_pgrande_iter1}
\resizebox{1\columnwidth}{!}{ $
\begin{aligned}
\hat{\tau}&_{i_{\bar{p}+j}}(\theta_{i_{\bar{p}+j}}) \leq \hat{\tau}_{i_{\bar{p}}}(\theta_{i_{\bar{p}}})+ \left( \hat{\tau}_{i_{max_{\{j,o\}}}}+\bar{d}_i \right) o\hat{L}_i\hat{\rho}_i^{\bar{p}+1} \leq \hat{\tau}_{i_{\bar{p}}}(\theta_{i_{\bar{p}}}) \\
&+ \left( \hat{\tau}_{i_{\bar{p}}}(\theta_{i_{\bar{p}}})+ \left( \hat{\tau}_{i_{max_{\{l,2o\}}}}+\bar{d}_i \right)o\hat{L}_i\hat{\rho}_i^{\bar{p}+1} +\bar{d}_i \right) o\hat{L}_i\hat{\rho}_i^{\bar{p}+1},
\end{aligned} $}
\end{equation}
where $l=j-\bar{p}$. \\
Thus, we can derive the following iterative expression to compute an over-estimate of the simulation error bound, where the considered horizon is denoted by $p=\ell\bar{p}+j$, with $\ell,j\in\mathbb{N}$ and $j\in[1,\bar{p})$:
\begin{equation}
\label{eq:tau_arx_pgrande_iter}
\begin{aligned}
\hat{\tau}_{i_{\ell\bar{p}+j}}&(\theta_{i_{\ell\bar{p}+j}}) \leq \hat{\tau}_{i_{\bar{p}}}(\theta_{i_{\bar{p}}}) \left( 1+\chi_{i,\bar{p}}+\chi_{i,\bar{p}}^2+\hdots+\chi_{i,\bar{p}}^{\ell-1} \right) +\\
&+ \bar{d}_i \left( \chi_{i,\bar{p}}+\chi_{i,\bar{p}}^2+\hdots+\chi_{i,\bar{p}}^{\ell} \right) + \tau_{i_{max_{\{j,\ell o\}}}}\chi_{i,\bar{p}}^\ell
\end{aligned}
\end{equation}
where
$\chi_{i,\bar{p}}=o\hat{L}_i\hat{\rho}_i^{\bar{p}+1},$
and
$\tau_{i_{max_{\{j,\ell o\}}}}=\max\{\hat{\tau}_{i_{j-\ell o}}(\theta_{i_{j-\ell o}}), \hdots, \hat{\tau}_{i_j}(\theta_{i_j}) \}.$

In general, the over-estimate   \eqref{eq:tau_arx_pgrande_iter} may diverge as $\ell$ increases. The next result provides a condition on $o,\,\hat{L}_i,\,\hat{\rho}_i,$ and $\bar{p}_i$ to guarantee convergence:
\begin{theorem}
\label{th:conv_tau_inf_ARX}
Consider any $\theta_{i_1}$ such that $h_{p,o}(\theta_{i_1}) \in \Gamma_{i_p}, \; \forall p \in [2,\bar{p}]$. Define  $\hat{\tau}_{i_{\infty}}$ as:
\begin{equation}
\label{eq:tau_inf_ARX_def}
\hat{\tau}_{i_{\infty}}(\theta_{i_{\bar{p}}}) \doteq \hat{\tau}_{i_{\bar{p}}}(\theta_{i_{\bar{p}}}) \left( \frac{1}{1-\chi_{i,\bar{p}}} \right) + \bar{d}_i \left( \frac{\chi_{i,\bar{p}}}{1-\chi_{i,\bar{p}}} \right).
\end{equation}
Then,
\begin{equation}
\label{eq:thm2_convergence}
\hat{\tau}_{i_p}\xrightarrow{p \to \infty}\hat{\tau}_{i_{\infty}}\iff o\hat{L}_i\hat{\rho}_i^{\bar{p}+1} <1
\end{equation}
\end{theorem}

\begin{proof} See the Appendix. \end{proof}

\begin{remark}
The convergence condition of Theorem \ref{th:conv_tau_inf_ARX} depends on $o$, $\hat{L}_i$ and $\hat{\rho}_i$, obtained using Procedures \ref{p:d_bar_est_procedure}-\ref{p:decay_est_procedure}, which in turn depend on the system at hand, on the collected data, and on $\bar{p}$, which  is chosen by the user during the identification procedure. Therefore, for given values of $o$, $\hat{L}_i, \, \hat{\rho}_i$, the value of $\bar{p}$ should be chosen large enough to satisfy \eqref{eq:thm2_convergence}.
\end{remark}

Assuming that condition \eqref{eq:thm2_convergence} is met, then the quantity $\hat{\tau}_{i_{\infty}}$ \eqref{eq:tau_inf_ARX_def} is the wanted infinite-horizon simulation error bound. The next results provide further insight on the bound \eqref{eq:tau_inf_ARX_def} and, in particular, on whether convergence of $\hat{\tau}_{i_p}$ to $\hat{\tau}_{i_{\infty}}$ is from above or below.

If condition \eqref{eq:thm2_convergence} holds, we can compute the difference $\hat{\tau}_{i_{\infty}}(\theta_{i_{\bar{p}}})-\hat{\tau}_{i_{\ell\bar{p}+j}}(\theta_{i_{\ell\bar{p}+j}})$ by means of truncated geometric series, leading to:
\begin{equation}
\label{eq:dist_tau_inf_ARX}
\begin{aligned}
&\hat{\tau}_{i_{\infty}}(\theta_{i_{\bar{p}}})-\hat{\tau}_{i_{\ell\bar{p}+j}}(\theta_{i_{\ell\bar{p}+j}}) = \\
&\; \; =\hat{\tau}_{i_{\bar{p}}}(\theta_{i_{\bar{p}}}) \left( \frac{\chi_{i,\bar{p}}^\ell}{1-\chi_{i,\bar{p}}} \right) + \bar{d}_i \left( \frac{\chi_{i,\bar{p}}^{\ell+1}}{1-\chi_{i,\bar{p}}} \right)  -\tau_{i_{max_{\{ j, \ell o \} }}} \chi_{i,\bar{p}}^\ell.
\end{aligned}
\end{equation}
The terms multiplying $\hat{\tau}_{i_{\bar{p}}}$ and $\bar{d}_i$ converge to their limit (see \eqref{eq:tau_inf_ARX_def}) from below, while the term $\tau_{i_{max_{\{ j, \ell o \} }}} \chi_{i,\bar{p}}^\ell$ converges to zero from above as $\ell\to \infty$. Thus, in general it is possible that $\hat{\tau}_{i_{\ell\bar{p}+j}} > \hat{\tau}_{i_{\infty}}$ for some $\ell$ and $j$. The next Lemma is concerned with this aspect.
\begin{lemma}
\label{l:overshoot_cond}
Let $\tau_{i_{max}}$ be defined as:
\[
\tau_{i_{max}}=\max \{\hat{\tau}_{i_1}(\theta_{i_1}),\hdots,\hat{\tau}_{i_{\bar{p}}}(\theta_{i_{\bar{p}}})\}.
\]
If $h_{p,o}(\theta_{i_1}) \in \Gamma_{i_p}, \; \forall p \in [2,\bar{p}]$, and if
\begin{equation}
\label{eq:lemma_1_condition}
\tau_{i_{max}} < \frac{\hat{\tau}_{i_{\bar{p}}}+\bar{d}_i \chi_{i,\bar{p}}}{1-\chi_{i,\bar{p}}},
\end{equation}
then $\hat{\tau}_{i_{\ell\bar{p}+j}}(\theta_{i_{\ell\bar{p}+j}}) \leq \hat{\tau}_{i_{\infty}}(\theta_{i_{\bar{p}}}), \; \forall \ell,j$. Otherwise, there could exist at least a pair $(\ell,j)$ such that $\hat{\tau}_{i_{\ell\bar{p}+j}}(\theta_{i_{\ell\bar{p}+j}})>\hat{\tau}_{i_{\infty}}(\theta_{i_{\bar{p}}})$.
\end{lemma}
\begin{proof} See the Appendix. \end{proof}
Note that the condition \eqref{eq:lemma_1_condition} of Lemma \ref{l:overshoot_cond} is in a sense adverse to the convergence condition \eqref{eq:thm2_convergence} of Theorem \ref{th:conv_tau_inf_ARX}. Since $\hat{\rho}_i<1$ by definition, there exists always a value of $\bar{p}$ large enough to satisfy \eqref{eq:thm2_convergence}. On the other hand, the right-hand side of \eqref{eq:lemma_1_condition} decreases as $\bar{p}$ increases, while $\tau_{i_{max}}$ is only weakly dependent on $\bar{p}$. Thus, Lemma \ref{l:overshoot_cond} suggests to pick a ``small'' value of $\bar{p}$, while Theorem \ref{th:conv_tau_inf_ARX} is generally satisfied with ``large'' $\bar{p}$. If one is interested in finding a finite simulation time such that, for any larger horizon, the simulation error bound converges from below to the infinite-horizon value, then the following result can be exploited.

\begin{remark}
\label{rm:end_of_overshoot}
Assume condition \eqref{eq:lemma_1_condition} is not satisfied, and take a small increase on the value of the asymptotic error bound, e.g. given by $\delta_i=10^{-2}\cdot \hat{\tau}_{i_{\infty}}(\theta_{i_{\bar{p}}})$. Define $\bar{\ell}$ as:
\[ \begin{array}{c}
\bar{\ell}=\min\limits_{\ell}\ell\\
\text{subject to}\\
\left\vert \frac{\hat{\tau}_{i_{\bar{p}}}+\bar{d}_i \chi_{i,\bar{p}}}{1-\chi_{i,\bar{p}}} - \tau_{i_{max}} \right\vert \chi_{i,\bar{p}}^{\bar{\ell}} < \delta_i
\end{array} \]
Then, as a straightforward consequence of Lemma \ref{l:overshoot_cond}, the following result holds:
\begin{equation}
\label{eq:l2_cond}\nonumber
\hat{\tau}_{i_{\ell\bar{p}+j}}(\theta_{i_{\ell\bar{p}+j}})\leq \hat{\tau}_{i_{\infty}}(\theta_{i_{\bar{p}}})+\delta_i, \; \forall \ell \geq \bar{\ell}, \; \forall j
\end{equation}
\end{remark}
%\begin{lemma}
%\label{l:end_of_overshoot}
%Assume condition \eqref{eq:lemma_1_condition} is not satisfied, and take a small increase on the value of the asymptotic error bound, e.g. given by $\delta_i=10^{-2}\cdot \hat{\tau}_{i_{\infty}}(\theta_{i_{\bar{p}}})$. Define $\bar{\ell}$ as:
%\[ \begin{array}{c}
%\bar{\ell}=\min\limits_{\ell}\ell\\
%\text{subject to}\\
%\left\vert \frac{\hat{\tau}_{i_{\bar{p}}}+\bar{d}_i \chi_{i,\bar{p}}}{1-\chi_{i,\bar{p}}} - \tau_{i_{max}} \right\vert \chi_{i,\bar{p}}^{\bar{\ell}} < \delta_i
%\end{array} \]
%Then, the following result holds:
%\begin{equation}
%\label{eq:l2_cond}
%\hat{\tau}_{i_{\ell\bar{p}+j}}(\theta_{i_{\ell\bar{p}+j}})\leq \hat{\tau}_{i_{\infty}}(\theta_{i_{\bar{p}}})+\delta_i, \; \forall \ell \geq \bar{\ell}, \; \forall j
%\end{equation}
%\end{lemma}
%\begin{proof} See the Appendix. \end{proof}

Finally, we show that Theorem \ref{th:conv_tau_inf_ARX} is also instrumental to derive a sufficient condition for the parameter vector $\theta_1$ to yield an asymptotically stable model.
\begin{theorem}
	\label{th:asymptotic_stability}
	Let Assumptions \ref{as:bounded_dist}-\ref{as:info_data} hold, and further assume that the chosen value of $\bar{p}$ satisfies \eqref{eq:thm2_convergence}. Consider a generic parameter vector $\theta_{i_1}\in\mathbb{R}^{o(m+1)}$.  If  
	\[
	h_{p,o}(\theta_{i_1}) \in \Gamma_{i_p}, \; \forall p \in [2,\bar{p}],
	\]
then the corresponding ARX model \eqref{eq:1s_pred_classic_form} is asymptotically stable.
\end{theorem}
\begin{proof}
	See the Appendix.
\end{proof}

Summing up, the findings and procedures  described so far address points 1)-3) and 5) of \textbf{Problem 1}. In the next section, we present two approaches to identify the one-step-ahead model \eqref{eq:1s_pred_classic_form} exploiting these results, thus dealing also with point 4) of \textbf{Problem 1}. 
\section{Predictor identification}
\label{s:pred_ident}
\subsection{Method I}
\label{ss:pred_id_method_1}
In this first approach, the parameters are estimated as:
\begin{subequations}\label{eq:opt_prob_method_1}
\begin{eqnarray}
&\hat{\theta}_{i_1}= \text{arg} \min\limits_{\theta_{i_1}}{\left\Vert \boldsymbol{\tau}_i(\theta) \right\Vert_{\infty}} \label{eq:opt_prob_method_1_a}\\
&\text{subject to} \nonumber\\
&h_{p,o}(\theta_{i_1}) \in \Theta_{i_p}^{L \rho}, \; \forall p \in [1,\bar{p}]
\label{eq:opt_prob_method_1_b}
\end{eqnarray}
\end{subequations}
where $\boldsymbol{\tau}_i=\left[ \hat{\tau}_{i_1}(\theta) \; \hat{\tau}_{i_2}(\theta) \; \hdots \; \hat{\tau}_{i_{\bar{p}}}(\theta) \right]^T$, $\hat{\tau}_{i_p}(\theta)$ is defined as in \eqref{eq:tau_hat_global_def}. %In \eqref{eq:opt_prob_method_1_b}, recall that $h_{p,o}(\theta_{i_1})=\theta_{i_p}$, i.e. these constraints pertain to the parameters of each  multi-step predictor, for $p=1,\ldots,\bar{p}$, deriving from the iteration of a one-step-ahead model. 
Namely, we thus aim to minimize the worst global error bound among all the simulation steps of interest, while ensuring that the resulting multi-step predictors comply with the derived FPSs. Problem \eqref{eq:opt_prob_method_1} is equivalent to: 
\begin{equation}
\label{eq:opt_prob_method_1_large}
\resizebox{1\columnwidth}{!}{ $
\begin{aligned}
\hat{\theta}_{i_1}=&\text{arg} \min\limits_{\theta_{i_1}}{\left( \max_{p \in [1,\bar{p}]^{\phantom{0}}}{\max_{k=1,\hdots,N^{\phantom{0}}}{\max_{\theta \in \Theta_{i_p}^{L\rho}}{ \left\vert \tilde{\varphi}_{i_p}(k)^T(\theta-\theta_{i_p}) \right\vert + \hat{\bar{\varepsilon}}_{i_p} }}} \right)} \\
&\text{subject to} \\
&h_{p,o}(\theta_{i_1}) \in \Theta_{i_p}^{L\rho}, \; \forall p \in [1,\bar{p}]
\end{aligned} $}
\end{equation}
This can be reformulated into a simpler optimization problem. The first step is to split the absolute value in the cost function of \eqref{eq:opt_prob_method_1_large} into two terms, by introducing the following quantities:
\begin{equation*}
\check{\varphi}_{i_p}(k)= \begin{cases} \tilde{\varphi}_{i_p}(k) \qquad \qquad \text{if} \quad k\leq N \\ -\tilde{\varphi}_{i_p}(k-N) \quad \text{if} \quad k>N  \end{cases} \, \text{for} \; \, k=1,\hdots,2N.
\end{equation*}
Then, let us define:
\begin{equation}\label{eq:preliminary_LP}
c_{i_{k_p}} \doteq \max_{\theta \in \Theta_{i_p}^{L\rho}}{\check{\varphi}_{i_p}(k)^T \theta, \quad k=1,\hdots,2N, \quad p=1,\hdots,\bar{p}}.
\end{equation}
The values of $c_{i_{k_p}}$ are computed by solving $2N\bar{p}$ linear programs (LPs). Then, \eqref{eq:opt_prob_method_1_large} can be reformulated as:\small
%\begin{equation}
%\label{eq:opt_prob_method_1_medium}
%\begin{aligned}
%\hat{\theta}_{i_1}= &\text{arg} \min_{\theta}{ \max_{p \in [1,\bar{p}]^{\phantom{0}}}{ \max_{k=1,\hdots,2N^{\phantom{0}}}{ \left( c_{i_{k_p}} - \check{\varphi}_{i_p}(k)^T \theta_{i_p} \right) }}} \\
%&\text{subject to} \\
%&h_{p,o}(\theta_{i_1})  \in \Theta_{i_p}^{L\rho}, \; \forall p \in [1,\bar{p}]
%\end{aligned}
%\end{equation}
%Finally, \eqref{eq:opt_prob_method_1_medium} corresponds to:
\begin{subequations}
\label{eq:opt_prob_method_1_small}
\begin{eqnarray}
&\hat{\theta}_{i_1} = \text{arg} \min\limits_{\theta_{i_1},\zeta}{\zeta} \label{eq:opt_prob_method_1_small_a}&\\
&\text{subject to} \nonumber&\\
&c_{i_{k_p}}- \check{\varphi}_{i_p}(k)^T h_{p,o}(\theta_{i_1}) \leq \zeta, \;\forall k \in [1,2N],\; \forall p \in [1,\bar{p}] \label{eq:opt_prob_method_1_small_b}&\\
&\theta_{i_1} \in \Theta_{i_1}^{L\rho} \label{eq:opt_prob_method_1_small_c}&\\
&h_{p,o}(\theta_{i_1}) \in \Theta_{i_p}^{L\rho}, \; \forall p \in [2,\bar{p}]& \label{eq:opt_prob_method_1_small_d}
\end{eqnarray}
\end{subequations}\normalsize
\eqref{eq:opt_prob_method_1_small} is a nonlinear optimization program (NLP) with linear cost \eqref{eq:opt_prob_method_1_small_a}, $2N\bar{p}$ nonlinear constraints \eqref{eq:opt_prob_method_1_small_b} (that require the preliminary solution of the $2N\bar{p}$ LPs \eqref{eq:preliminary_LP}), $2N$ linear constraints \eqref{eq:opt_prob_method_1_small_c}, finally $2N(\bar{p}-1)$ nonlinear constraints \eqref{eq:opt_prob_method_1_small_d}. All nonlinear constraints are polynomial, thus Jacobian and Hessians can be efficiently computed analytically and exploited in the numerical solver.

A possible alternative is to use a quadratic cost in \eqref{eq:opt_prob_method_1}, e.g. $\boldsymbol{\tau}_i(\theta)^TQ\boldsymbol{\tau}_i(\theta)$ where $Q$ is a symmetric positive definite weighting matrix. This would penalize a weighted average of the simulation error bounds over the considered horizon $\bar{p}$, instead of its worst-case as done in \eqref{eq:opt_prob_method_1}. In this case, a similar reformulation can be carried out, resulting in an NLP with quadratic cost and linear and polynomial constraints.

\subsection{Method II}
\label{ss:pred_id_method_2}
In the second approach, we search the one-step-ahead model that minimizes a standard simulation error criterion,
 while enforcing membership to the FPS $\Theta_{i_1}^{L\rho}$ and  the exponentially decaying behavior of the iterated predictors parameters for $p>1$ up to $\bar{p}$.  The corresponding  NLP is:
\begin{subequations}
\label{eq:opt_prob_method_2}
\begin{eqnarray}
\hat{\theta}_{i_1}=&\text{arg} \min\limits_{\theta_{i_1} \in \Theta_{i_1}^{L\rho}}{\left\Vert \tilde{\boldsymbol{Y}}_i-\hat{\boldsymbol{Z}}_i(\theta_{i_1}) \right\Vert_2^2} \\
&\text{subject to} \nonumber\\
&h_{p,o}(\theta_{i_1}) \in \Gamma_{i_p}, \; \forall p \in [2,\bar{p}]\label{eq:opt_prob_method_2b}
\end{eqnarray}
\end{subequations}
where
\[ \resizebox{1\columnwidth}{!}{ $\begin{aligned}
&\tilde{\boldsymbol{Y}}_i=\left[ \tilde{y}_i(1) \; \tilde{y}_i(2) \; \hdots \; \tilde{y}_i(N) \right]^T\\
&\hat{\boldsymbol{Z}}_i(\theta_{i_1})=\left[\tilde{\varphi}_{i_1}(0)^T \theta_{i_1} \; \tilde{\varphi}_{i_2}(0)^T h_{2,o}(\theta_{i_1}) \; \hdots \; \tilde{\varphi}_{i_N}(0)^T h_{N,o}(\theta_{i_1}) \right]^T. 
\end{aligned}$}
\]
Problem \eqref{eq:opt_prob_method_2} is a NLP with polynomial cost function, $2N$ linear constraints and $2(\bar{p}-1)$ polynomial constraints. Also in this method, Jacobian and Hessians can be efficiently computed analytically.

\begin{remark}
\label{rm:id_algo}
The optimization problems of Methods I and II are always feasible by construction. The inclusion of constraints \eqref{eq:opt_prob_method_1_b} and \eqref{eq:opt_prob_method_2b} guarantees consistency with the results of Section \ref{s:ms_sm_appr}, including asymptotic stability of the derived models, as shown by Theorem \ref{th:asymptotic_stability}. %Methods I and II are two viable identification approaches exploiting the results presented in this work, pertaining to the Set Membership framework. 
It is also possible to adopt variations, e.g. by adding more constraints to Method II, like $\theta_{i_p}\in \Theta_{i_p}^{L\rho}$ for some selected $p\in[1,p_{max}]$. %Another possibility is a variation of the prediction error method, where the one-step-ahead prediction error minimization is combined with constraints \eqref{eq:opt_prob_method_1_b} or \eqref{eq:opt_prob_method_2b}.
\end{remark}

\subsection{Computational aspects}\label{ss:comput_aspects}
Computational effort is often the main drawback in Set Membership identification. The optimization problems \eqref{eq:opt_prob_method_1_small} and \eqref{eq:opt_prob_method_2} are constrained Nonlinear Programs (NLP), thus they are not convex in general. Finding a feasible point for this class of problems can be computationally hard, even when this point is guaranteed to exist like in our case. In our tests in Section \ref{s:results}, the NLPs are solved resorting to  Sequential Quadratic Programming (SQP) algorithms (MatLab's  \verb|fmincon|). In the literature (e.g., \cite{nocedal2006numerical}) the guaranteed global convergence of SQP to a local minimizer has been proven, under rather mild assumptions. Yet, in applications these assumptions are not easy to verify. What we can however observe are the practical performance obtained with such a well-established numerical approach. Given the non-convex nature of the NLPs, for each problem instance we ran the solver several times, each one with a different initialization value, to evaluate whether it gave consistent results and to choose the best local optimum among the resulting ones. In particular, in all the runs for either NLPs \eqref{eq:opt_prob_method_1_small} or \eqref{eq:opt_prob_method_2} (around 200 for each problem and for both the numerical example and the real-world application in Section \ref{s:results}), the SQP algorithm was always able to converge to a feasible local minimizer.\\
%
%
% Due to the non-convex nature of the problem, the numerical solver generally converges to a local optimum. The sub-optimality of the solution clearly depends on the initialization of the parameters in the SQP solver. One possible approach to attain the global optimum is to run several times the algorithm with different initialization values, e.g. generating them randomly within a certain set $\mathcal{P}$, and then choose among the obtained solutions the model giving the lowest cost function value.
The complexity of the NLP mainly depends on the FPSs, which are used to define the constraints and to compute the simulation error bounds.
The FPSs are polytopes whose number of facets generally grows linearly with the number of data points, and whose dimensionality grows linearly with the horizon $p$ in the multi-step approach adopted here. To reduce complexity, in the literature there are several contributions proposing to outer-approximate the FPSs, see references provided in Section \ref{ss:FPS_and_tau_def}. These approaches present different trade-offs between complexity reduction and conservativeness.\\
An alternative we prefer in our context, where computational time is not critical, since the identification is carried out off-line, is to resort to a redundant constraint removal procedure. In Method I and II, the set membership constraints are nonlinear in the optimization variable $\theta_{i_1}$. However, for each $p$ the corresponding FPS features $2N$ inequalities that are linear in the entries of $\theta_{i_p}$. Therefore, we can split \[
h_{p,o}(\theta_{i_1}) \in \Theta_{i_p}, \; \forall p,
\]
into 
\[\theta_{i_p} \in \Theta_{i_p} \; \wedge \theta_{i_p}=h_{p,o}(\theta_{i_1}), \; \forall p,
\]
and then carry out a redundant constraint removal routine on each set of linear constraints  $\theta_{i_p} \in \Theta_{i_p}$ for $p\in[1,\bar{p}]$. 

In \cite{paulraj_comparative}, a comparative analysis of different redundant constraints identification approaches is presented. We tried in our tests the one described in \cite{caron_mcdonald_ponic}, which is based on the minimization of each linear constraint function, subject to the remaining constraints. Then, if the obtained optimal value is positive, the constraint is marked as redundant and can be removed from the set. %of the cost function is compared with the right-hand side of the corresponding constraint: if the  to decide whether the latter is redundant or not (\textit{i.e.} if it is lower of the right-hand side value or not, respectively). Notice that this algorithm never consider the objective function of the original optimization problem. 
This method requires the solution of as many LPs  as the number of original constraints, for each FPS. %Moreover, we can remove only the constraints that are redundant for each individual $p$, as we are not able to identify whether a constraint is redundant for the all stack of $\theta_{i_p}(\theta_{i_1})$, with $p \in [1,\bar{p}]$. 
However, in our tests it consistently reduces the total number of constraints we are dealing with. %, which appear as nonlinear ones in the optimization problems presented in Section \ref{s:pred_ident}. The same procedure can be also applied to reduce the number of constraints \eqref{eq:opt_prob_method_1_small_b} in Method I. %This solution could help to reduce the overall computational effort, with respect to the usage of the full number of original nonlinear constraints, especially if $N$ and $\bar{p}$ are sufficiently large.
\section{State-space predictor form}
\label{s:ss_form}
When the state is measurable, we can identify a predictor model of the form \eqref{eq:sist_desc}, where $C$ is replaced by the identity matrix. Therefore, the $p$-steps-ahead $i$-th output is:
\begin{equation}
\label{eq:output_ss_form}
z_i(k+p)=x_i(k+p)=C_i A^px(k)+C_i \sum_{j=1}^p A^{j-1}Bu(k+p-j),
\end{equation}
where $C_i$ is the $i$-th row of the identity matrix. We form the regressor $\psi_p \in \mathbb{R}^{n+mp}$ as:
\begin{equation}
\label{eq:regr_ss_def}\nonumber
\psi_p(k)= \left[ x(k)^T \; u(k)^T \; u(k+1)^T \; \cdots \; u(k+p-1)^T \right]^T,
\end{equation}
and the parameter vector $\theta^0_{i_p} \in \mathbb{R}^{n+mp}$ is:
\begin{equation}
\label{eq:param_vect_ss_form}\nonumber
\theta^0_{i_p}= \begin{bmatrix} C_iA^p & C_iA^{p-1}B & C_iA^{p-2}B & \hdots & C_iAB & C_iB \end{bmatrix}^T.
\end{equation}
Then, \eqref{eq:output_ss_form} can be written as $z_i(k+p)=\psi_p(k)^T \theta^0_{i_1}$. \\
Note that, differently from the ARX form considered in the previous sections, the regressor is now the same for all the $n$ output equations. The noise-corrupted measure of the system state is $y(k)=z(k)+d(k)$. We define the one-step-ahead model as:
\begin{equation}
\label{eq:pred_1s_ss_form}
\hat{z}_i(k+1)=\varphi_1(k)^T \theta_{i_1},
\end{equation}
where $\varphi_1(k)=\left[y(k)^T \; u(k)^T \right]^T \in \mathbb{R}^{n+m}$, and $\theta_{i_1}=\left[ C_i A \; C_i B \right]^T \in \mathbb{R}^{n+m}$. Then, the multi-step predictors are obtained by iteration of \eqref{eq:pred_1s_ss_form}, and their parameters are polynomial functions of the parameters of the one-step-ahead predictor, denoted as $\theta_{i_p}=h_{p,n}(\theta_{i_1})\in \mathbb{R}^{n+mp}$. \\
Under Assumptions \ref{as:asympt_stable}-\ref{as:bounded_dist}, the regressor $\psi_p$ belongs to a compact set $\Psi_p$:
\begin{equation*}
\psi_p(k)\in \Psi_p \subset \mathbb{R}^{n+mp}, \; \Psi_p \text{ compact, } \forall p \in \mathbb{N}, \; \forall k \in \mathbb{Z},
\end{equation*}
and $\varphi_p$ belongs to a compact set $\Phi_p$:
\begin{equation*}
\varphi_p(k) \in \Phi_p = \Psi_p \oplus \mathbb{D}_p, \; \forall p \in \mathbb{N}, \; \forall k \in \mathbb{Z},
\end{equation*}
where $\mathbb{D}_p \doteq \left\{ \left[ d^T,0,\hdots,0 \right]^T : |d| \leq \bar{d}_0 \right\}$. \\
The sampled data set is defined as:
\begin{equation*}
\tilde{\mathscr{V}}_{i_p}^N \doteq \left\{ \tilde{v}_{i_p}(k)= \begin{bmatrix} \tilde{\varphi}_p(k) \\ \tilde{y}_{i_p}(k) \end{bmatrix}, \; k=1,\hdots,N \right\} \subset \mathbb{R}^{1+n+mp},
\end{equation*}
with $\tilde{y}_{i_p}(k) \doteq \tilde{y}_i(k+p)$, and its continuous counterpart is:
\begin{equation*} \resizebox{1\columnwidth}{!}{ $
\mathscr{V}_{i_p} \doteq \left\{ v_{i_p}= \begin{bmatrix} \varphi_p \\ y_{i_p} \end{bmatrix}: \, y_{i_p} \in Y_{i_p}(\varphi_p), \, \forall \varphi_p \in \Phi_p \right\} \subset \mathbb{R}^{1+n+mp}, $}
\end{equation*}
where $Y_{i_p}(\varphi_p) \subset \mathbb{R}$ is the compact set of all the possible $i$-th output values corresponding to each regressor $\varphi_p \in \Phi_p$, and to every possible noise realization $d_i: |d_i|\leq \bar{d}_{0_i}$. Assumption \ref{as:info_data} and its consequences apply also here, as in Section \ref{s:multitep_pred}. Moreover, all the results presented in Section \ref{s:ms_sm_appr} can be straightforwardly extended to the case of the predictor defined in \eqref{eq:pred_1s_ss_form}. The main difference is that here the statement of Corollary \ref{c:lambda_rate_conv} becomes $\lambda_{i_p} = \bar{d}_0^T \left\vert (C_iA^p)^T \right\vert \leq \left\Vert \bar{d}_0 \right\Vert_1 L_i \rho_i^{p+1}$, and thus \eqref{eq:L_eff_arx} becomes $\hat{L}_i=\nicefrac{L_i'}{\left\Vert \bar{d} \right\Vert_1}$.

Furthermore, also the results presented in Section \ref{s:tau_inf_intro} can be extended to the predictor model defined by \eqref{eq:pred_1s_ss_form}. Here, going through the same reasoning of \eqref{eq:err_bound_arx_pgrande}-\eqref{eq:tau_arx_pgrande_iter1} leads to:
\begin{equation}
\label{eq:err_bound_ss_pgrande}\nonumber
\begin{aligned}
&\left\vert z_i(k+\ell\bar{p}+j) - \hat{z}_i(k+\ell\bar{p}+j) \right\vert \leq \hat{\tau}_{i_{\ell\bar{p}+j}}(\theta_{i_{\ell\bar{p}+j}}) \leq \\
&\; \; \leq \hat{\tau}_{i_{\bar{p}}}(\theta_{i_{\bar{p}}}) \sum_{m=0}^{\ell-1} \chi_{i,\bar{p}}^m + \left\Vert \bar{d} \right\Vert_1 \sum_{m=1}^\ell \chi_{i,\bar{p}}^m + \left\Vert \hat{\tau}_j \right\Vert_1 \chi_{i,\bar{p}}^\ell,
\end{aligned}
\end{equation} 
where $\hat{\tau}_j=\left[\hat{\tau}_{1_j}(\theta_{1_j}), \, \hdots, \, \hat{\tau}_{n_j}(\theta_{n_j}) \right]^T$ and $\chi_{i,\bar{p}}=\hat{L}_i \hat{\rho}_i^{\bar{p}+1}$. \\
Theorem \ref{th:conv_tau_inf_ARX} and the related Remarks and Lemmas apply straightforwardly to \eqref{eq:pred_1s_ss_form}, with minor modifications: the convergence condition of Theorem \ref{th:conv_tau_inf_ARX} is here given by $\left\vert \chi_{i,\bar{p}} \right\vert = \left\vert \hat{L}_i \hat{\rho}_i^{\bar{p}+1} \right\vert < 1$, and \eqref{eq:tau_inf_ARX_def} becomes
\[
\hat{\tau}_{i_{\infty}}(\theta_{i_{\bar{p}}})=\hat{\tau}_{i_{\bar{p}}}(\theta_{i_{\bar{p}}}) \left( \frac{1}{1-\chi_{i,\bar{p}}} \right) + \left\Vert \bar{d} \right\Vert_1 \left( \frac{\chi_{i,\bar{p}}}{1-\chi_{i,\bar{p}}} \right).
\]
Lemma \ref{l:overshoot_cond} and Remark \ref{rm:end_of_overshoot} still apply, but here $\tau_{i_{max}}$ has to be replaced with $\left\Vert \hat{\tau}_j \right\Vert_1$. \\
Finally, $\hat{\theta}_{i_1}$ is identified resorting to the methods presented in Section \ref{s:pred_ident}, and the estimated system matrices $\hat{A}\approx A$ and $\hat{B}\approx B$ are built as:
\[
\begin{array}{l}
\hat{A}=
\left[\begin{array}{c}
\hat{\theta}_{1_1}^{(1:n)}\\
\vdots\\
\hat{\theta}_{n_1}^{(1:n)}\\
\end{array}\right]
,\;
\hat{B}=
\left[\begin{array}{c}
	\hat{\theta}_{1_1}^{(n+1:m)}\\
	\vdots\\
	\hat{\theta}_{n_1}^{(n+1:m)}\\
\end{array}\right],
\end{array}
\]
where $\hat{\theta}_{i_1}^{(j:l)}$ denotes the elements of vector $\hat{\theta}_{i_1}$ from the $j$-th entry to the $l$-th entry.

\section{Simulation and experimental results}\label{s:results}
\subsection{Simulation results}
\label{s:sim_res}
We first assess the performance of the proposed identification procedure in a numerical example, and compare the results with those of  established identification approaches: the prediction error method (PEM), and the simulation error method (SEM). PEM approach identifies the model parameters by minimizing the squared $\ell_2$-norm of the one-step-ahead prediction error. SEM approach is based on the minimization of the squared $\ell_2$-norm of the simulation error, where the simulation of the system output is obtained by iteration of the prediction model, and corresponds to the unconstrained version of \eqref{eq:opt_prob_method_2}. More details can be found e.g. in \cite{soderstrom1989system} or \cite{tomita1992equation}. The numerical example analyzed here also gives insight on the procedures proposed in Section \ref{ss:estim_d_o_rho}. 

We consider the following one input, three outputs underdamped asymptotically stable system in continuous time $t$:
\begin{equation}
\label{eq:underdamp_sys_ss}
\begin{aligned}
\dot{x}(t) &= \begin{bmatrix} 0 & 0 & -160 \\ 1 & 0 & -24 \\ 0 & 1 & -10.8 \end{bmatrix} x(t)+\begin{bmatrix} 160 \\ 0 \\ 0 \end{bmatrix} u(t) \\
y(t)&=x(t)+d(t)
\end{aligned}
\end{equation}
\begin{figure}[thpb]
	\centering
	\includegraphics[width=1.0\columnwidth]{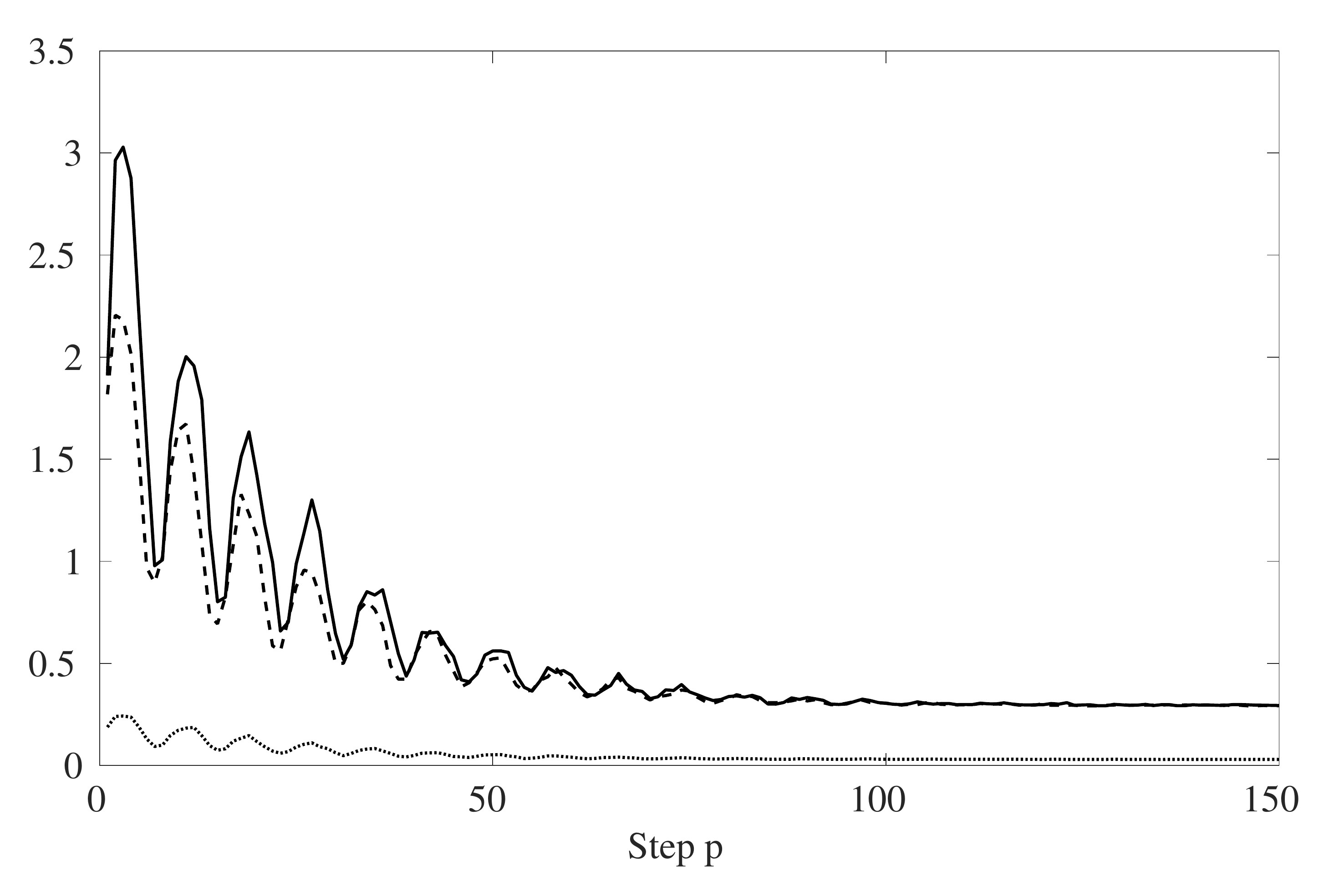}
	\caption{Numerical example: estimated values of $\underline{\lambda}_{i_p}$ with $\bar{d}=\left[0.7 \; 0.7 \; 0.07 \right]^T$ for the ARX predictor case. Solid line: $\underline{\lambda}_{1_p}$; dashed line: $\underline{\lambda}_{2_p}$; dotted line: $\underline{\lambda}_{3_p}$.}
	\label{f:lambda_for_est_d_ARX}
\end{figure}
The system eigenvalues are: $s_1=-10$ and $s_{2,3}=-0.4 \pm i 3.98$, and the output measurements are affected by uniformly distributed random noise, with $\bar{d}_{0}=[1\;1\;0.1]^T$. The input takes value in the set $\{-1; \; 0; \; 1 \}$ randomly every 4 time units. The considered data set is composed of 10000 input and output data points collected with a sampling frequency of 10 samples per time unit. The first half of the data set is used for the identification phase, while the second half is used for validation. 

%\begin{figure}[thpb]
%	\centering
%	\includegraphics[width=1.02\columnwidth]{lambda_SS_d07007}
%	\caption{Numerical example: estimated values of $\underline{\lambda}_{i_p}$ with $\bar{d}=\left[0.7 \; 0.7 \; 0.07 \right]^T$ for the state-space predictor case. Solid line: $\underline{\lambda}_{1_p}$; dashed line: $\underline{\lambda}_{2_p}$; dotted line: $\underline{\lambda}_{3_p}$.}
%	\label{f:lambda_for_est_d_SS}
%\end{figure} 
%
%
\begin{figure}[thpb]
	  \centering
      		\begin{tabular}{c}
     			\includegraphics[width=0.94\columnwidth]{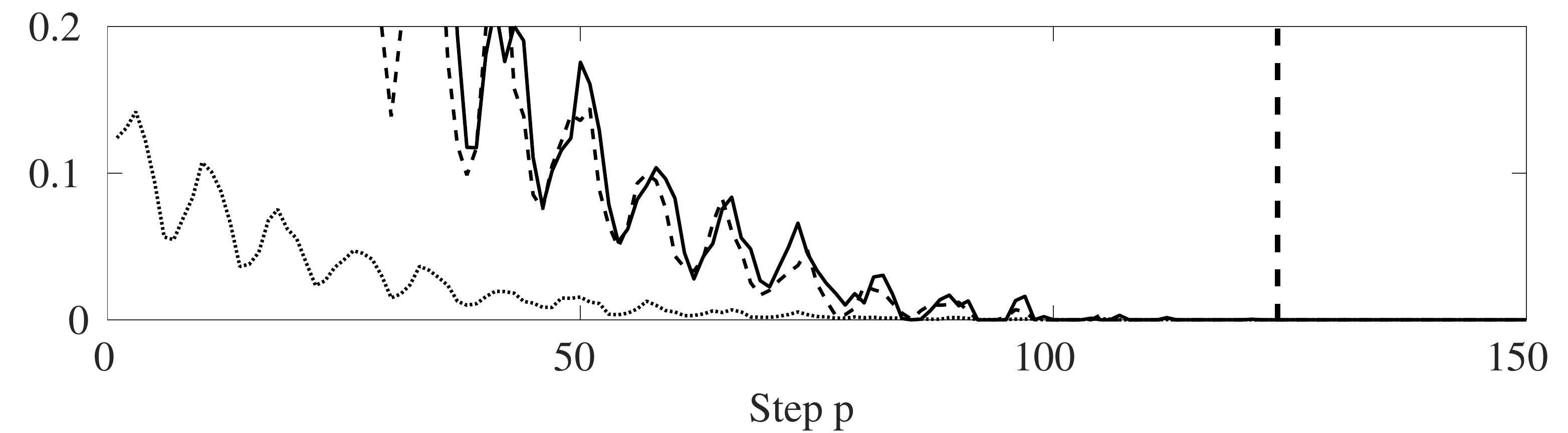} \\
     			\includegraphics[width=0.94\columnwidth]{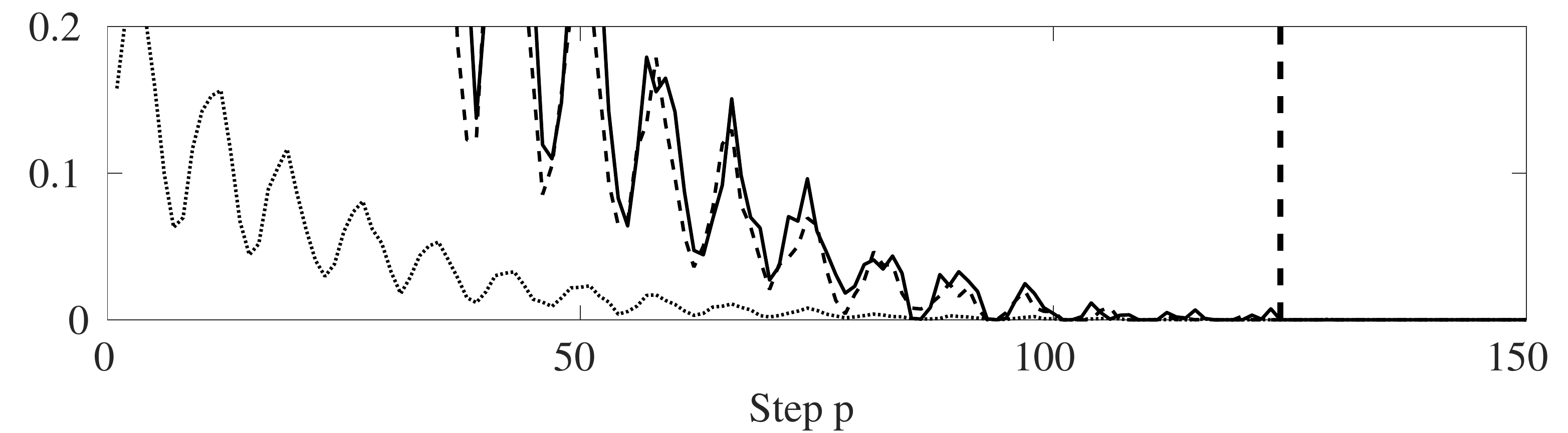} \\
     			\includegraphics[width=0.94\columnwidth]{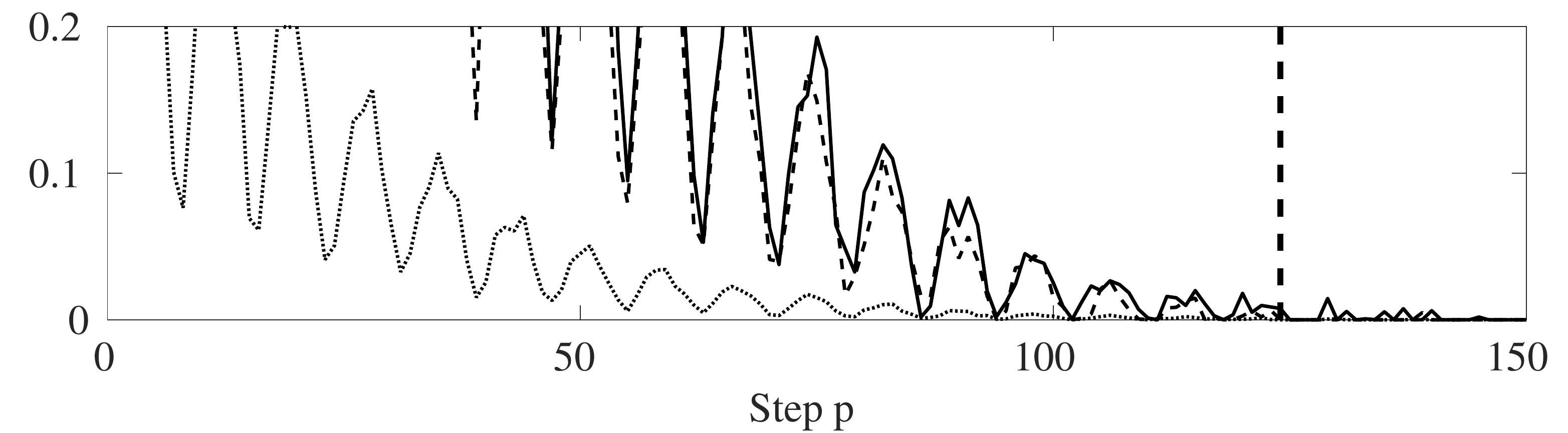}
      		\end{tabular}
      		\begin{tikzpicture}[baseline,overlay]
      			\node[font=\color{black}] at (-0.2,2.79) {\small (a)};
      			\node[font=\color{black}] at (-0.2,0.365) {\small (b)};
      			\node[font=\color{black}] at (-0.2,-2.05) {\small (c)};
      		\end{tikzpicture}
      \caption{Numerical example: estimated values of $\underline{\lambda}_{i_p}$ with $\bar{d}=\left[1 \; 1 \; 0.1 \right]^T$ for the ARX predictor case. Solid line: $\underline{\lambda}_{1_p}$; dashed line: $\underline{\lambda}_{2_p}$; dotted line: $\underline{\lambda}_{3_p}$. Fig. (a): $o=4$; fig. (b): $o=3$; fig. (c): $o=2$. The dashed vertical lines indicate the value of $\bar{p}=\max\limits_i{\bar{p}_i}$ obtained for the chosen $\bar{d}$, and are used to set $o$ at the lowest possible value such that all $\underline{\lambda}_{i_p}=0, \, \forall p>\bar{p}$.}
      \label{f:lambda_ARX_fun_o}
\end{figure}
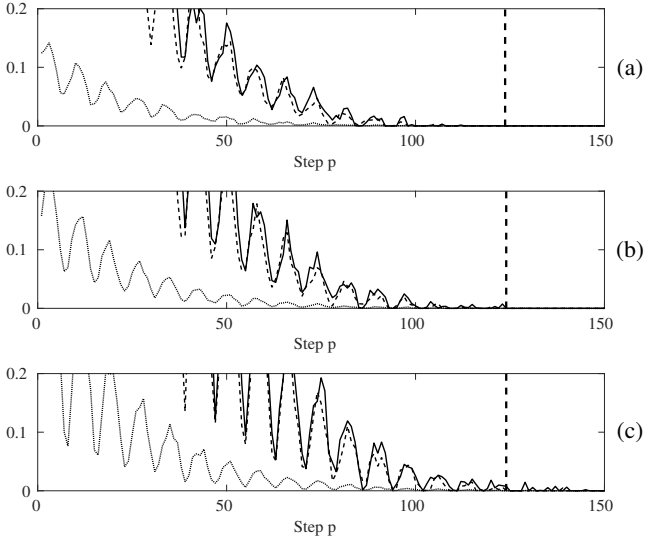 
To carry out a complete analysis, we consider both the ARX model formulation and the state-space one. In Procedure \ref{p:d_bar_est_procedure}, we start with a guess of the noise bound  $\bar{d}=\left[0.7 \; 0.7 \; 0.07 \right]^T$, and compute the corresponding values of $\underline{\lambda}_{i_p}$, for $p \in [1,150]$, resorting to \eqref{eq:lambda_p_calc}. The results are depicted in Fig. \ref{f:lambda_for_est_d_ARX} for the ARX case; a similar behavior is obtained for the state-space model. As predicted by Theorem \ref{th:conv_lambda_diff_d}, $\underline{\lambda}_p$ converges to $\left[ 0.3 \; 0.3 \; 0.03 \right]^T$, which corresponds to $\bar{d}_{0_i}-\bar{d}_i$. Then, we set the noise bound to $\bar{d}=\left[1 \; 1 \; 0.1 \right]^T$, which is indeed consistent with the real one.
Fig. \ref{f:lambda_ARX_fun_o} depicts the results of Procedure \ref{p:o_est_procedure}. It correctly indicates $o=3$ as the minimum model order of the ARX predictors.
Then, we carry out Procedure \ref{p:decay_est_procedure} to estimate the parameters $\hat{L}_i$ and $\hat{\rho}_i$ of the exponentially decaying trend, see \eqref{eq:est_L_rho_optprob} and \eqref{eq:L_eff_arx}. For the ARX predictor, the resulting parameters are $\hat{L}=\left[3.094 \; 2.162 \; 0.259 \right]^T$ and $\hat{\rho}=\left[0.959 \; 0.959 \; 0.959 \right]^T$, while for the state-space case we obtain $\hat{L}=\left[3.982 \; 0.956 \; 0.092 \right]^T$ and $\hat{\rho}=\left[0.961 \; 0.965 \; 0.961 \right]^T$. Fig. \ref{f:Lr_ARX} shows the estimated decay bounds over the corresponding values of $\underline{\lambda}_{i_p}$ for the ARX model structure. Similar results are obtained for the state-space structure.
\begin{figure}[thpb]
	\centering
	\includegraphics[width=1.0\columnwidth]{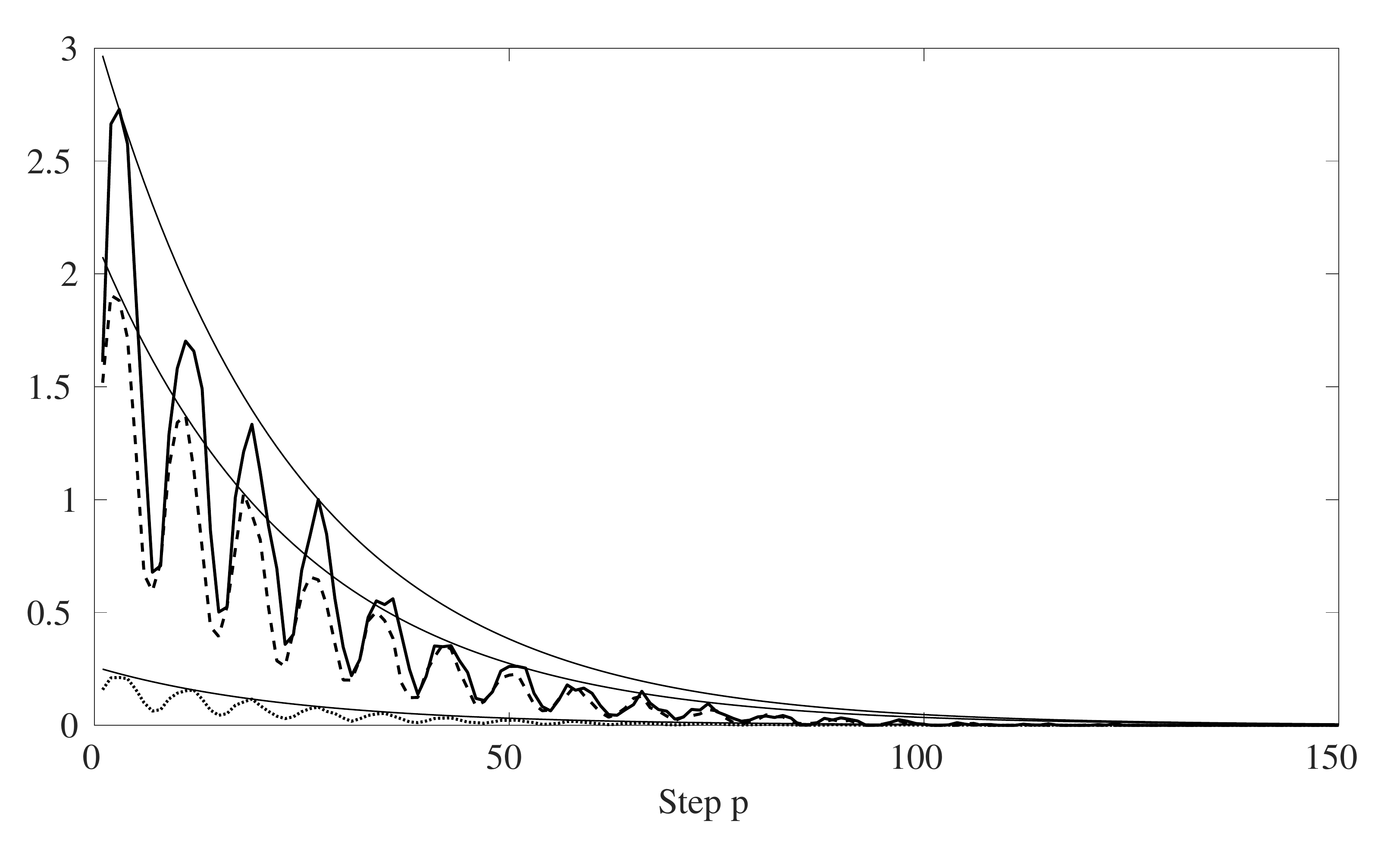}
	\caption{Numerical example: estimated values of $\underline{\lambda}_{i_p}$ and of the corresponding bound $\hat{L}_i\hat{\rho}_i^p$ for the ARX predictor case. Solid line: $\underline{\lambda}_{1_p}$; dashed line: $\underline{\lambda}_{2_p}$; dotted line: $\underline{\lambda}_{3_p}$. The exponentially decaying bounds are represented with thin continuous lines which lie over the corresponding $\underline{\lambda}_{i_p}$.}
	\label{f:Lr_ARX}
\end{figure} 
%%
%\begin{figure}[thpb]
%	\centering
%	\includegraphics[width=1.02\columnwidth]{Lr_SS}
%	\caption{Numerical example: estimated values of $\underline{\lambda}_{i_p}$ and of the corresponding bound $\hat{L}_i\hat{\rho}_i^p$ for the state-space predictor case. Solid line: $\underline{\lambda}_{1_p}$; dashed line: $\underline{\lambda}_{2_p}$; dotted line: $\underline{\lambda}_{3_p}$. The exponentially decaying bounds are represented with a thin continuous line which lies over the corresponding $\underline{\lambda}_{i_p}$.}
%	\label{f:Lr_SS}
%\end{figure} 
%
\begin{figure}[thpb]
	\centering
	\includegraphics[width=1.0\columnwidth]{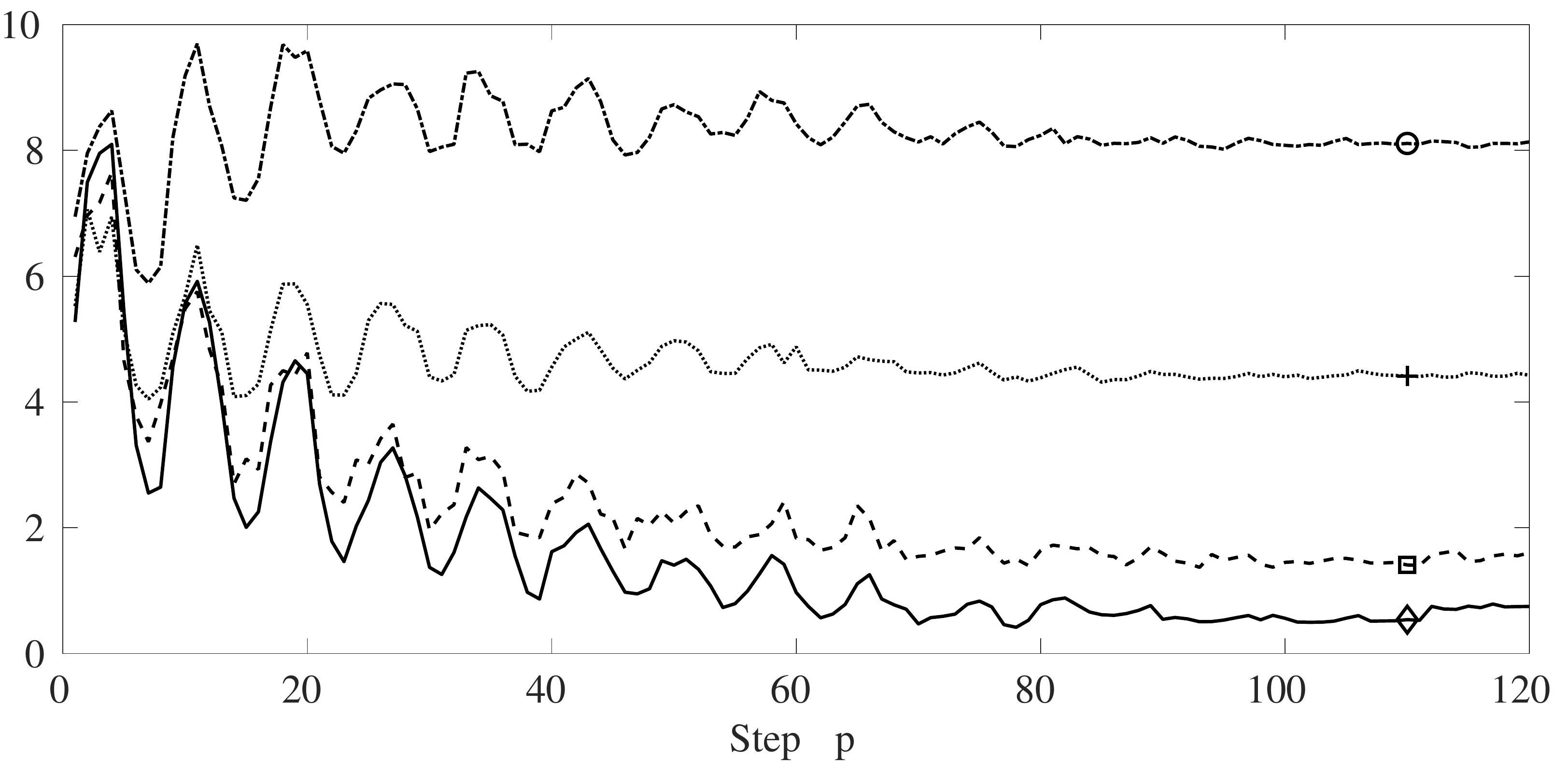}
	\caption{Numerical example: guaranteed simulation error bound $\hat{\tau}_{2_p}$ on $\hat{z}_2$ for the ARX predictor. Dotted line with `$+$': Method I; solid line with `$\diamond$': Method II; dashed line with `$\square$': SEM approach; dash-dot line with: `$\circ$': PEM approach.}
	\label{f:tau_G_ARX}
\end{figure}
%
%\begin{figure}[thpb]
%	\centering
%	\includegraphics[width=0.99\columnwidth]{Tau_G_SS_y2}
%	\caption{Numerical example: guaranteed simulation error bound $\hat{\tau}_{2_p}$ on $\hat{z}_2$ for the state-space predictors. Dotted line with `$+$': Method I; solid line with `$\diamond$': Method II; dashed line with `$\square$': SEM approach; dash-dot line with: `$\circ$': PEM approach.}
%	\label{f:tau_G_SS}
%\end{figure}
%
\begin{table}[ht]
\caption{\label{t:G_tau_err_ARX}Numerical example: guaranteed simulation error bound and worst-case prediction error on validation data for the ARX predictor.}
\centering
\setlength\tabcolsep{3.5pt}
\small
\begin{tabular}{c c c c c c c c c c}
\toprule
 & & \multicolumn{4}{c}{i=1} & \multicolumn{4}{c}{i=3} \\
\cmidrule(lr){3-6}\cmidrule(lr){7-10}
 & $p$: & $1$ & $8$ & $19$ & $27$ &$1$ & $12$ & $35$ & $50$ \\ 
\midrule
\multirow{2}{*}{SEM} & $\tau_{i_p}$ & 8.11 & \textbf{2.72} & 8.13 & 6.10 & 0.76 & 1.20 & 0.45 & \textbf{0.22} \\
 & $e_{i_p}$ & 4.11 & \textbf{1.90} & \textbf{4.06} & \textbf{3.27} & 0.46 & 0.61 & 0.30 & 0.19 \\
\midrule
\multirow{2}{*}{Method II} & $\tau_{i_p}$ & \textbf{6.26} & 5.03 & \textbf{7.36} & \textbf{5.92} & \textbf{0.79} & \textbf{0.91} & \textbf{0.40} & 0.24 \\
 & $e_{i_p}$ & \textbf{3.15} & 4.01 & 4.17 & 3.40 & \textbf{0.36} & \textbf{0.39} & \textbf{0.24} & \textbf{0.18} \\
\bottomrule
\end{tabular} 
\normalsize
\end{table}
\begin{table}[ht]
\caption{\label{t:G_tau_err_SS}Numerical example: guaranteed simulation error bound and worst-case prediction error on validation data for the state-space predictor.}
\centering
\setlength\tabcolsep{3.5pt}
\small
\begin{tabular}{c c c c c c c c c c}
\toprule
 & & \multicolumn{4}{c}{i=1} & \multicolumn{4}{c}{i=3} \\
\cmidrule(lr){3-6}\cmidrule(lr){7-10}
 & $p$: & $1$ & $12$ & $35$ & $50$ &$1$ & $8$ & $19$ & $27$ \\ 
\midrule
\multirow{2}{*}{SEM} & $\tau_{i_p}$ & 9.97 & 13.9 & 6.34 & 3.45 & 0.27 & 0.64 & 0.29 & 0.31 \\
 & $e_{i_p}$ & 4.58 & 8.85 & 3.21 & 2.47 & 0.21 & 0.38 & 0.25 & 0.27 \\
\midrule
\multirow{2}{*}{Method II} & $\tau_{i_p}$ & \textbf{6.45} & \textbf{7.55} & \textbf{3.41} & \textbf{2.21} & \textbf{0.27} & \textbf{0.31} & \textbf{0.16} & \textbf{0.13} \\
 & $e_{i_p}$ & \textbf{3.03} & \textbf{3.54} & \textbf{2.04} & \textbf{1.76} & \textbf{0.18} & \textbf{0.19} & \textbf{0.15} & \textbf{0.13} \\
\bottomrule
\end{tabular} 
\normalsize
\end{table}
\begin{figure}[thpb]
	\centering
	\includegraphics[width=0.99\columnwidth]{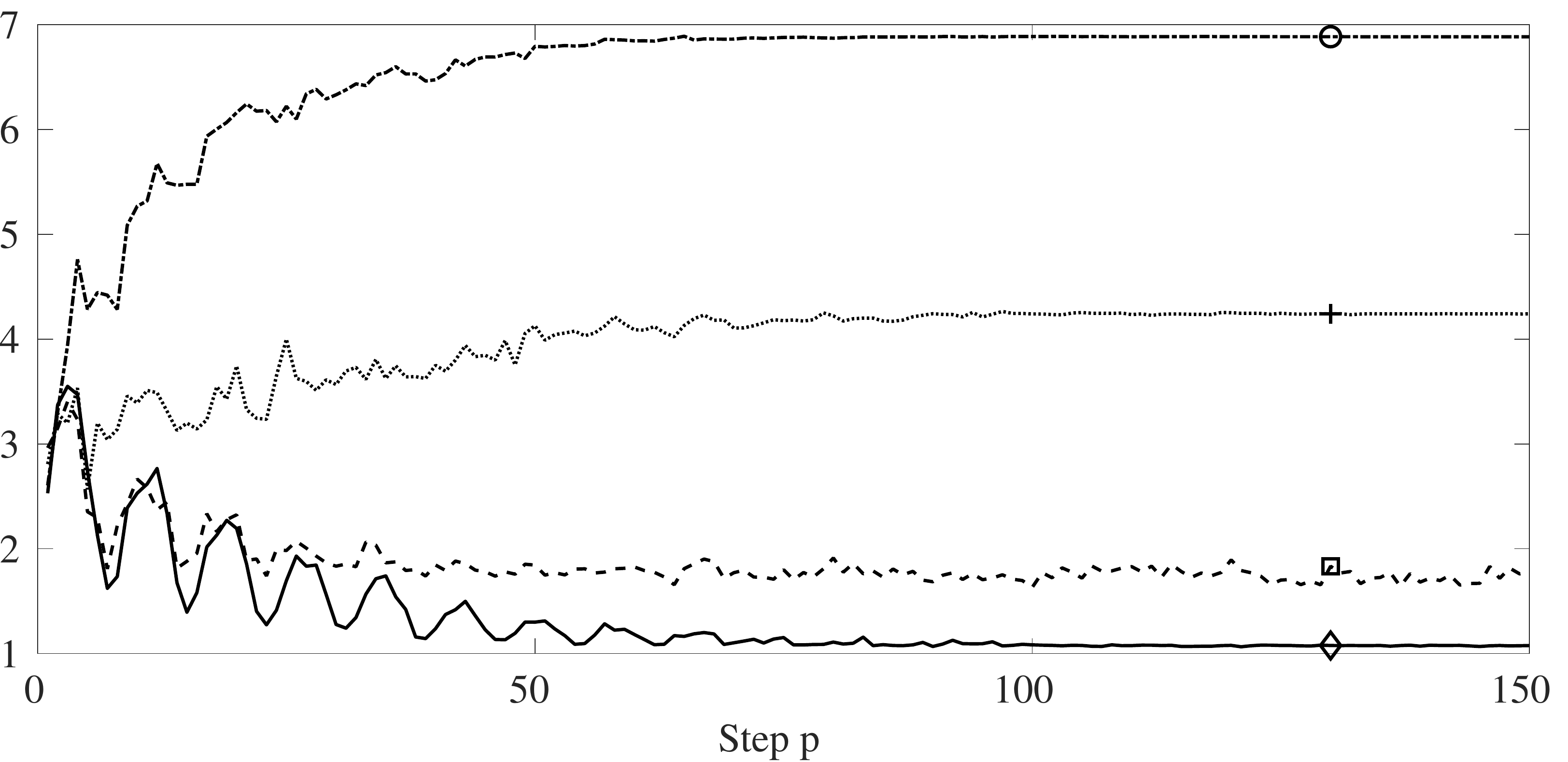}
	\caption{Numerical example: worst-case validation error $e_{2_p}$ on $\hat{z}_2$ for the ARX predictor. Dotted line with `$+$': Method I; solid line with `$\diamond$': Method II; dashed line with `$\square$': SEM approach; dash-dot line with: `$\circ$': PEM approach.}
	\label{f:err_G_ARX}
\end{figure}
%
%\begin{figure}[thpb]
%	\centering
%	\includegraphics[width=0.99\columnwidth]{ErrV_G_SS_y1}
%	\caption{Numerical example: validation error $e_{1_p}$ on $\hat{z}_1$ for the state-space predictor. Dotted line with `$+$': Method I; solid line with `$\diamond$': Method II; dashed line with `$\square$': SEM approach; dash-dot line with: `$\circ$': PEM approach.}
%	\label{f:err_G_SS}
%\end{figure}
%
%
The parameters of the predictors are eventually identified using Methods I and II, and the FPS are defined as in \eqref{eq:FPS_with_decay}, where $\hat{\bar{\varepsilon}}_{i_p}$ is obtained from $\underline{\lambda}_{i_p}$ with $\alpha=1.2$, see \eqref{eq:eps_hat_def}. As benchmark, we employ predictors identified using the PEM and SEM approaches. We compare the performance of the identified models in terms of guaranteed simulation error bounds $\hat{\tau}_{i_p}(\theta_{i_p})$, computed over the identification data set with $\gamma=1.1$, and of worst-case validation error, defined as:
\[
e_{i_p}=\max_{k=1,\hdots,N} \left\vert \tilde{y}_i(k+p)-\hat{z}_i(k+p) \right\vert
\]
and calculated over the validation data set. Fig. \ref{f:tau_G_ARX} depicts the obtained guaranteed error bounds related to the output $z_2$ for the identified ARX models, while Fig. \ref{f:err_G_ARX} presents the corresponding observed worst-case validation error.
%
%\begin{figure}[thpb]
%	\centering
%	\includegraphics[width=0.99\columnwidth]{Y2_sim_G_ARX}
%	\caption{Numerical example: simulated output $\hat{z}_2$ with the ARX predictor. Black solid line: measured output $\tilde{y}_2$; red solid line: real system output $z_2$; dashed line: simulated output with SEM predictor; dash-dotted line: simulated output with PEM predictor; dotted line: simulated output with Method II predictor.}
%	\label{f:Y_sim_G_ARX}
%\end{figure}
%%
%\begin{figure}[thpb]
%	\centering
%	\includegraphics[width=0.99\columnwidth]{Y2_sim_G_ARX_zoom}
%	\caption{Numerical example: simulated output $\hat{z}_2$ with the ARX predictor. Black solid line: measured output $\tilde{y}_2$; red solid line: real system output $z_2$; dashed line: simulated output with SEM predictor; dash-dotted line: simulated output with PEM predictor; dotted line: simulated output with Method II predictor.}
%	\label{f:Y_sim_G_ARX_zoom}
%\end{figure}
%
%
\begin{figure}[thpb]
	  \centering
      		\begin{tabular}{c}
     			 \small (a) \normalsize \\ \includegraphics[width=0.92\columnwidth]{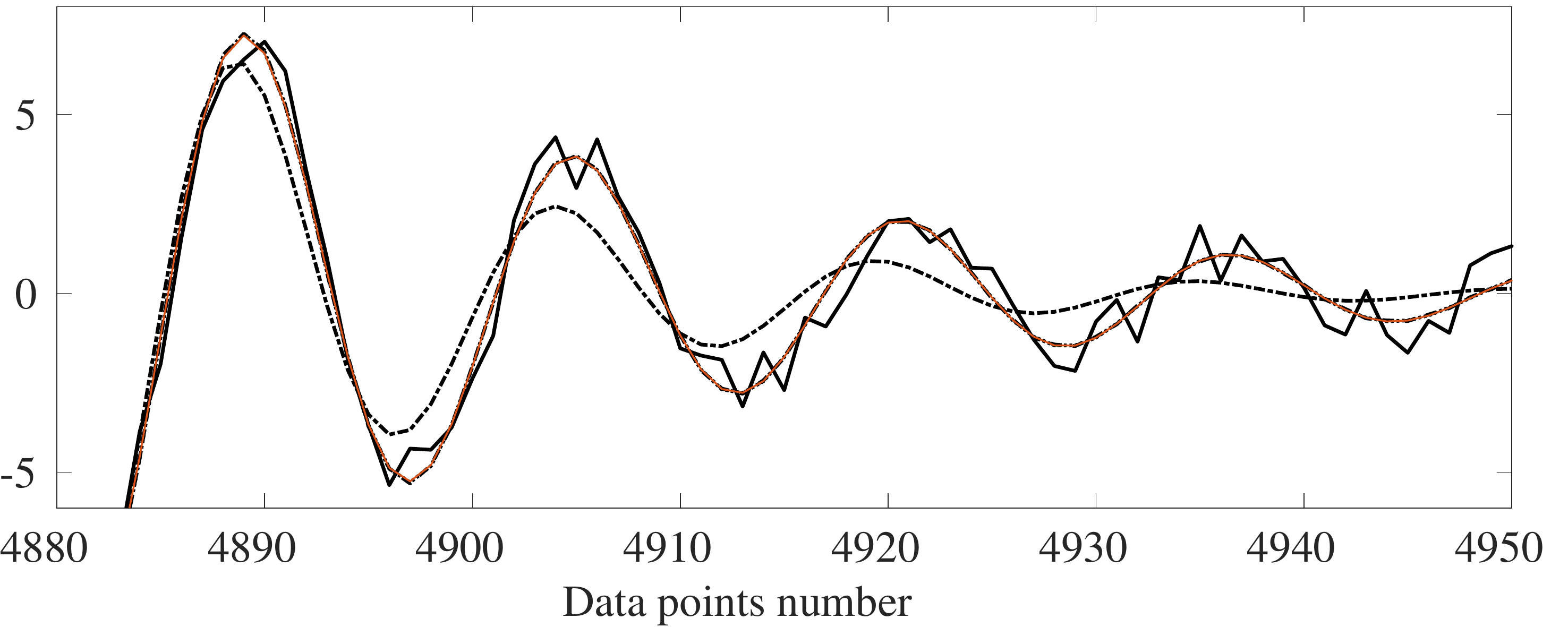} \\ 
     			 \small (b) \normalsize \\ \includegraphics[width=0.92\columnwidth]{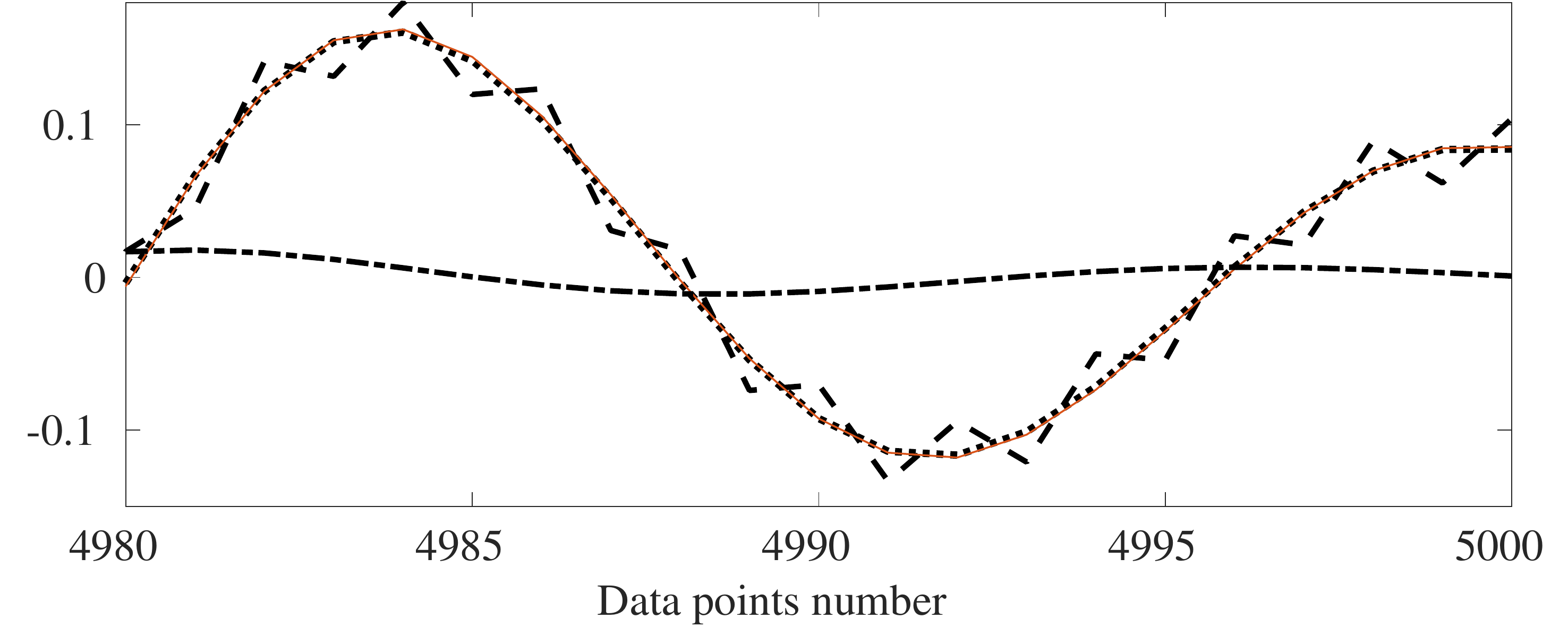}
      		\end{tabular}
      \caption{Numerical example, Fig. (a): simulated output $\hat{z}_2$ with the ARX predictor; Fig. (b): detailed view. Black solid line: measured output $\tilde{y}_2$; red solid line: real system output $z_2$; dashed line: simulated output with SEM predictor; dash-dotted line: simulated output with PEM predictor; dotted line: simulated output with Method II predictor.}
      \label{f:Y_sim_G_ARX}
\end{figure} 
It can be noted that the model identified with Method I achieves (as expected from the employed  cost criterion) the smallest worst-case (over $p$) guaranteed error bound, however at the cost of a higher guaranteed bound for longer horizon, as compared to Method II and SEM. Qualitatively similar outcomes are obtained for the other outputs and for the state-space model structure. More values of $\hat{\tau}_{i_p}$ and $e_{i_p}$ are reported in Tables \ref{t:G_tau_err_ARX} and \ref{t:G_tau_err_SS}. These results indicate that the proposed identification Method II has comparable, and often better, performance with respect to the SEM approach, in terms of both error bound and observed validation error, and overall better performance than the other two approaches. In particular, we notice that the predictor identified using Method II has good performance in long-range simulation, as the SEM approach, but also with better performance for short horizon values, outperforming the SEM. In particular, Fig. \ref{f:err_G_ARX} and Tables \ref{t:G_tau_err_ARX} and \ref{t:G_tau_err_SS} show how the predictor identified using Method II is able to provide small one-step-ahead prediction error, as the PEM approach, and small simulation error, as the SEM approach, combining the advantages of the two identification approaches. This is possible thanks to the constraints  $\theta_{i_1}\in \Theta_{i_1}^{L\rho}$ in \eqref{eq:opt_prob_method_2}, which are able to improve the performance over the SEM approach in terms of one-step-ahead prediction error. 
\begin{figure}[thpb]
	  \centering
      		\begin{tabular}{c}
     			 \small (a) \normalsize \\ \includegraphics[width=0.98\columnwidth]{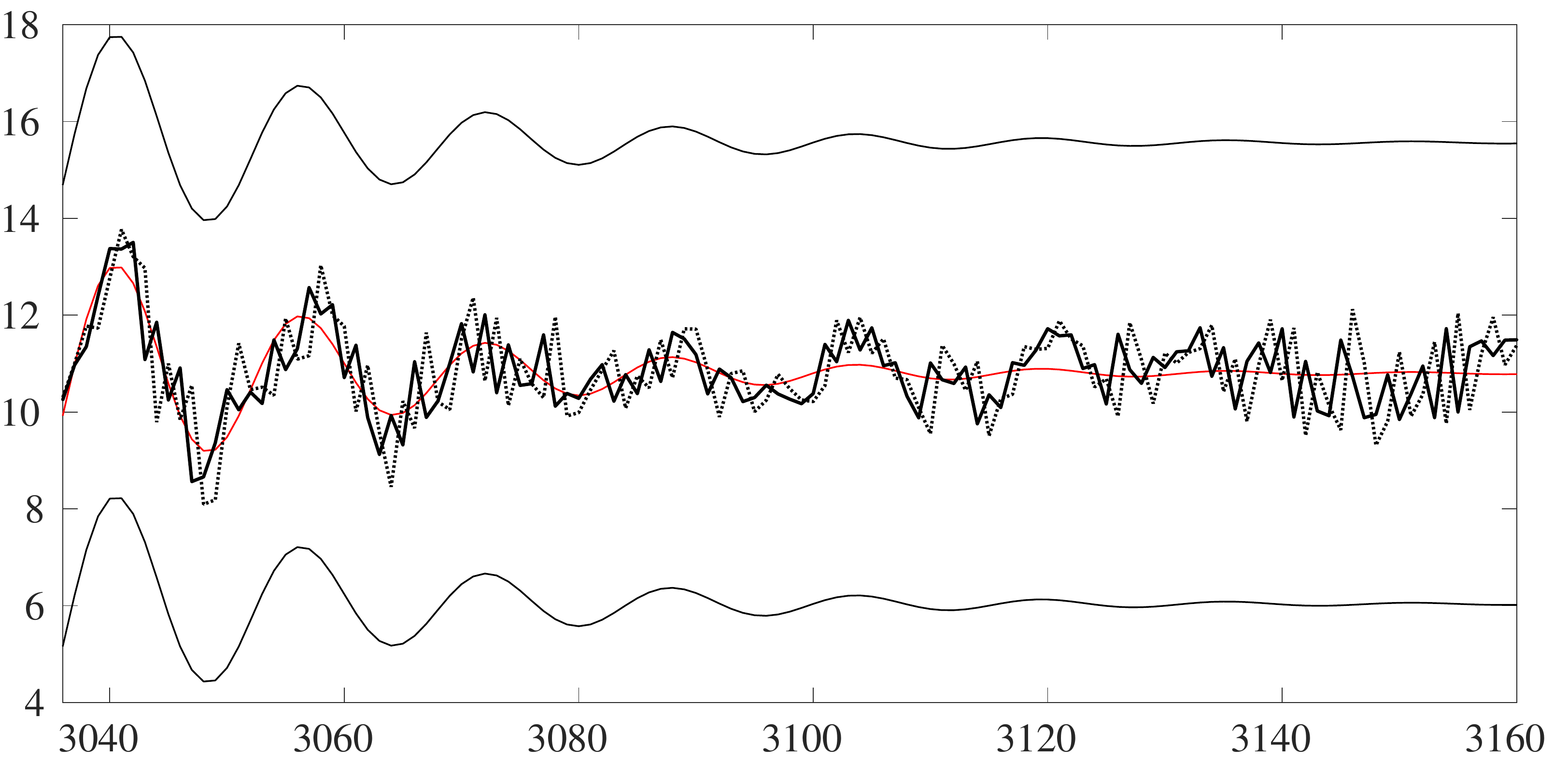} \\ 
     			 \small (b) \normalsize \\ \includegraphics[width=0.98\columnwidth]{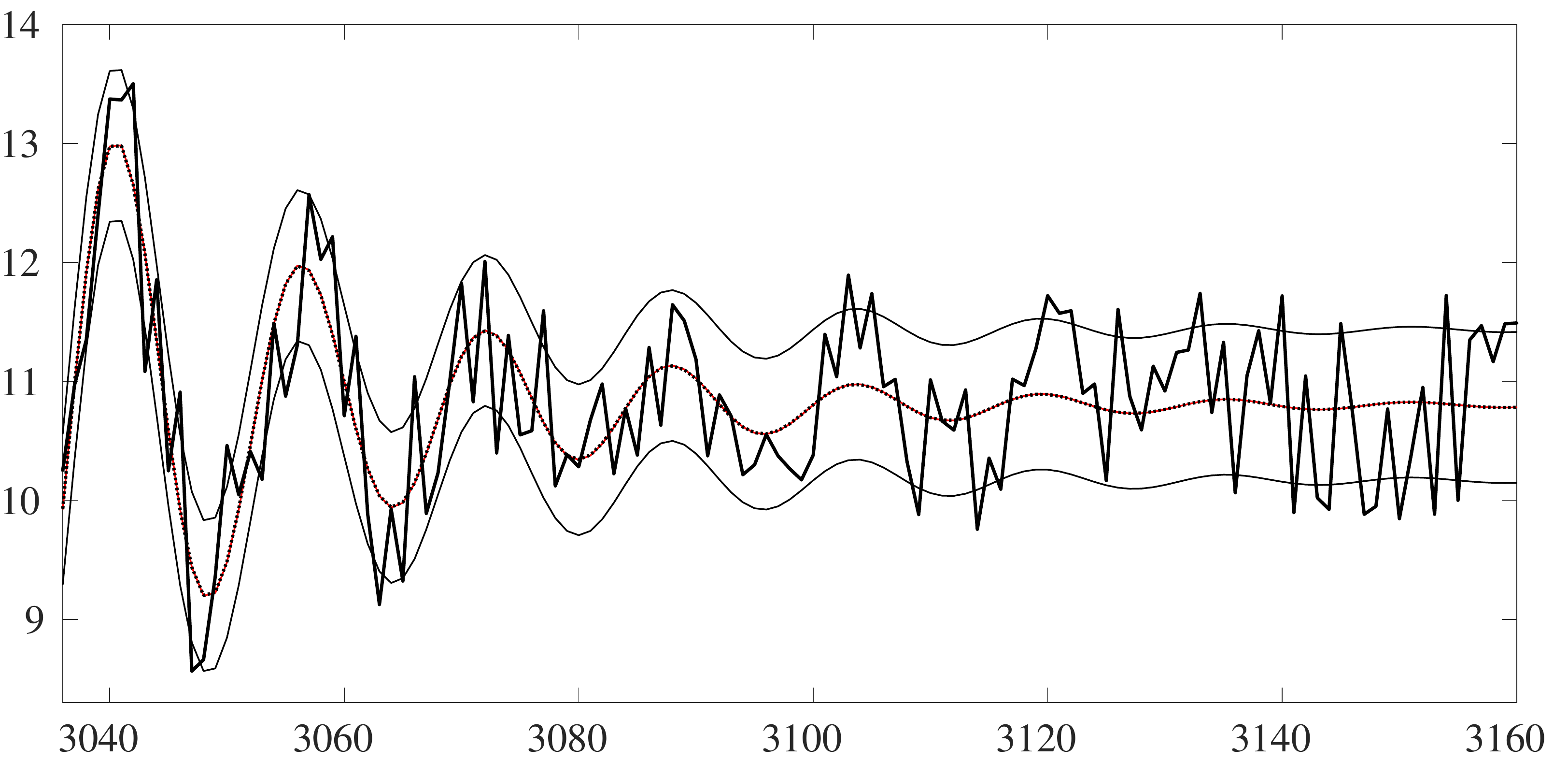}
      		\end{tabular}
      \caption{Numerical example, Fig. (a): one-step prediction of output $\hat{z}_2$ with the ARX predictor ($p=1$); Fig. (b): simulated output $\hat{z}_2$ with the ARX predictor. Black solid line: measured output $\tilde{y}_2$; red solid line: real system output $z_2$; dashed line: predicted/simulated output with Method II predictor; thin black lines: Method II predictor error bounds.}
      \label{f:Y_bound_G_ARX}
\end{figure}
The model identified using Method I, on the other hand, obtains a lower simulation error for short horizon with respect to the other approaches, at the cost of a higher simulation error for longer horizon. This stems from the fact that we are minimizing the worst-case error over the whole horizon. Using a quadratic cost in \eqref{eq:opt_prob_method_1_large} in order to minimize the average error, as commented in Remark \ref{rm:id_algo}, could partly improve this issue.

Besides the worst-case performance, Tables \ref{t:G_ARX_RMSE} and \ref{t:G_SS_RMSE} present exemplifying values of the root mean squared error (RMSE) for the predictors obtained using different identification methods, having respectively an ARX and a state-space formulation. The RMSE is calculated over the validation data set as:
\[
\text{RMSE}=\sqrt{\frac{\sum_{k=1}^N \Big( \tilde{y}_i(k+p)-\hat{z}_{i}(k+p) \Big)^2}{N}},
\]
i.e. it considers the $p$-steps-ahead simulation error. The results in the tables confirm the good  performance of Method II, since the obtained predictor yields better (for short horizon) or similar RMSE as compared with SEM. The predictor identified with Method I has good performance for short simulation horizons, but its error increases for longer ones. %These results indicate again that Method II outperforms the SEM approach, with a minor increase in computational complexity, thanks to the combination of good long-range simulation performance, coming from the SEM approach, and good one-step-ahead prediction performances, coming from the usage of the FPS $\Theta_{i_1}^{L\rho}$.
\begin{figure}[thpb]
	\centering
	\includegraphics[width=0.99\columnwidth]{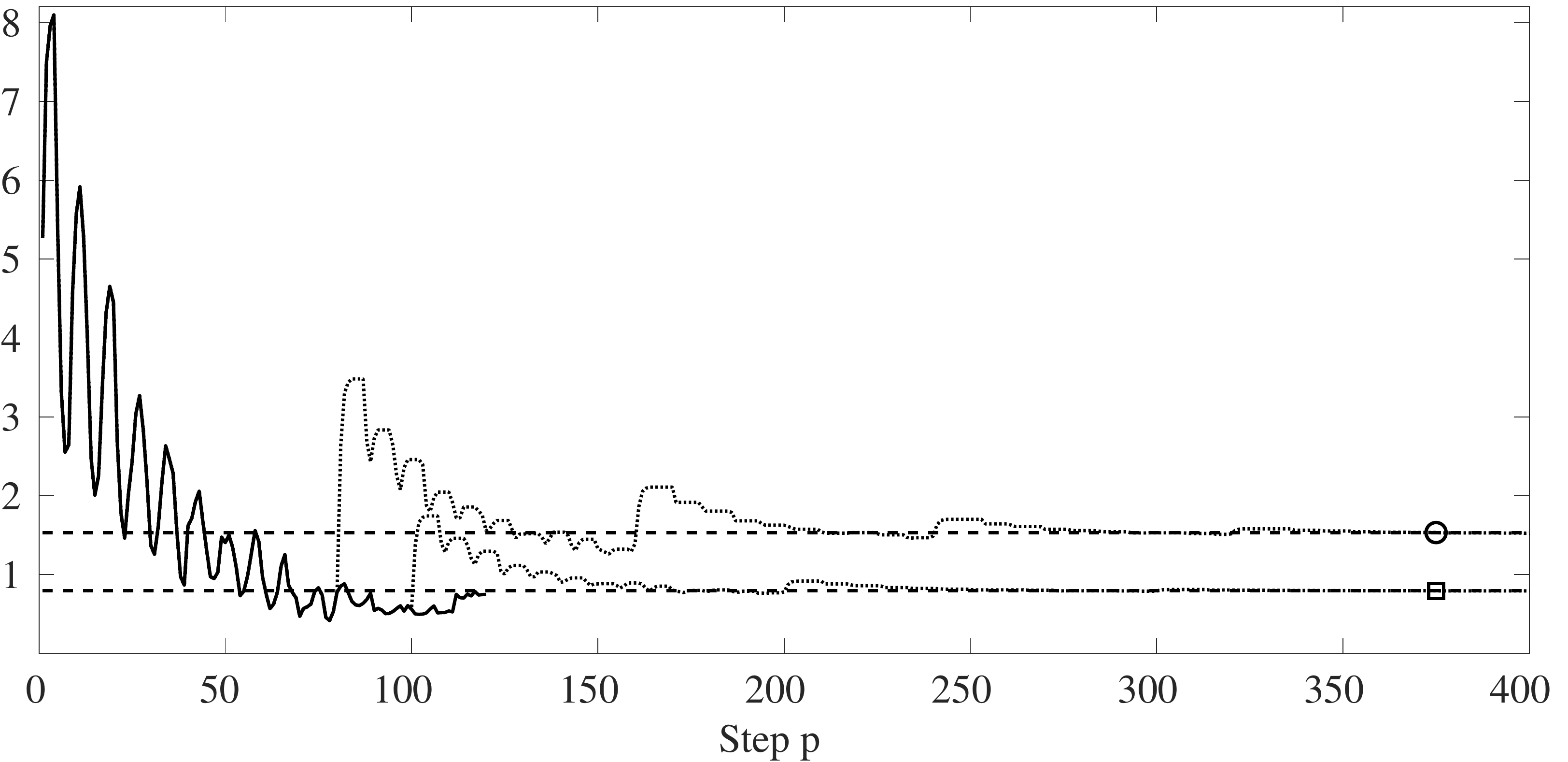}
	\caption{Numerical example: infinite-horizon error bound $\hat{\tau}_{2_p}$ for the ARX predictor identified using Method II. Solid line: bound calculated using \eqref{eq:tau_hat_global_def} for $p \in [1,120]$; dotted line with `$\circ$': iterative bound \eqref{eq:tau_arx_pgrande_iter} with $\bar{p}=80$; dashed line with `$\circ$': infinite-horizon bound \eqref{eq:tau_inf_ARX_def} with $\bar{p}=80$; dotted line with `$\square$': iterative bound with $\bar{p}=100$; dashed line with `$\square$': infinite-horizon bound with $\bar{p}=100$.}
	\label{f:tau_iter_G_ARX}
\end{figure}
\begin{table}[ht]
\caption{\label{t:G_ARX_RMSE}Numerical example, validation data: Root Mean Square Error for $p$-step-ahead prediction and simulation for ARX models.}
\centering
\setlength\tabcolsep{3.5pt}
\small
\begin{tabular}{c c c c c c c c}
\toprule
\multicolumn{2}{c}{RMSE} & $p=1$ & $p=10$ & $p=20$ & $p=30$ & $p=60$ & sim \\ 
\midrule
 & $y_1$ & 5.539 & 21.42 & 26.32 & 27.77 & 30.27 & 30.56 \\
PEM & $y_2$ & \textbf{0.930} & 1.287 & 1.636 & 1.775 & 1.923 & 1.937 \\ 
 & $y_3$ & 0.097 & 0.179 & 0.222 & 0.234 & 0.246 & 0.248 \\
\midrule
 & $y_1$ & 1.523 & 1.651 & 1.366 & \textbf{0.728} & \textbf{0.620} & \textbf{0.580} \\
SEM & $y_2$ & 1.018 & \textbf{0.935} & 0.787 & 0.667 & 0.661 & 0.577 \\ 
 & $y_3$ & 0.159 & 0.185 & 0.163 & 0.082 & 0.065 & \textbf{0.059} \\ 
\midrule
 & $y_1$ & \textbf{0.979} & 1.661 & 1.431 & 1.344 & 1.403 & 1.411 \\
Method I & $y_2$ & 0.946 & 0.987 & 0.983 & 1.016 & 1.148 & 1.176 \\ 
 & $y_3$ & \textbf{0.095} & \textbf{0.101} & \textbf{0.100} & 0.102 & 0.109 & 0.119 \\ 
\midrule
 & $y_1$ & 1.178 & \textbf{1.278} & \textbf{1.082} & 0.894 & 0.898 & 0.897 \\
Method II & $y_2$ & 0.978 & 0.941 & \textbf{0.750} & \textbf{0.589} & \textbf{0.577} & \textbf{0.573} \\ 
 & $y_3$ & 0.130 & 0.134 & 0.106 & \textbf{0.064} & \textbf{0.060} & \textbf{0.059} \\ 
\bottomrule
\end{tabular} 
\normalsize
\end{table}

Fig. \ref{f:Y_sim_G_ARX} presents an example of time-course of the system output $z_2$, comparing the real, measured and simulated values. In the detailed view of Fig. \ref{f:Y_sim_G_ARX} (b) it is possible to appreciate how the simulation obtained using Method II predictor overlaps the true system output $z_2$. Fig. \ref{f:Y_bound_G_ARX} displays an other example of time-course of the system output, comparing the real and measured values with the one-step-ahead prediction, Fig. \ref{f:Y_bound_G_ARX} (a), and with the long-range simulation, Fig. \ref{f:Y_bound_G_ARX} (b), reporting in both cases the corresponding error bounds. From Fig. \ref{f:Y_bound_G_ARX} (b) it is possible to notice that the guaranteed error bound for the long-range simulation case is smaller than the amplitude of the noise $d$. Thus, the distance of $\tilde{y}_2$ from $z_2$ is often greater then the error bound of $\hat{z}_2$. Fig. \ref{f:tau_iter_G_ARX} depicts the comparison between the simulation error bound $\hat{\tau}_{2_p}$ calculated using the definition \eqref{eq:tau_hat_global_def} for $p\in[1,120]$, the iterative error bound \eqref{eq:tau_arx_pgrande_iter} and the infinite-horizon error bound \eqref{eq:tau_inf_ARX_def}, obtained setting $\bar{p}=80$ and $\bar{p}=100$, for the case of the predictor having an ARX structure, identified using Method II. Here, it is possible to notice that the iterative and the infinite-horizon error bounds become a tighter upper-bound of $\hat{\tau}_{2_p}$, obtained from its definition, as $\bar{p}$ increases. \\
Fig. \ref{f:Tau_var_alpha} shows the effects of the choice of $\alpha$ in \eqref{eq:eps_hat_def} on the identification performance. Here, different values of $\alpha$ are used, %to account for the usage of a finite data-set, 
repeating the identification procedure using Method II, and computing the simulation error bound $\hat{\tau}_{2_p}$ for the obtained models. It is possible to see that for $\alpha=1$ the obtained FPS is too small, resulting in a validation error $e_{2_p}$ that violates the provided error bound, as motivated by Remark \ref{rm:bounds_conservativeness}. Moreover, we can see that, with a smaller $\alpha$, the constraint $\theta_{i_1}\in \Theta_{i_1}^{L\rho}$ provides a reduced error for short prediction horizons, at the price of an increase of the error for longer horizons, whereas a bigger value of $\alpha$ obtains the opposite effect. Table \ref{t:G_ARX_RMSE_var_alpha} presents the RMSE obtained by models identified using Method II with different values of $\alpha$. Here, it is possible to appreciate that a small increase of $\alpha$ reduces the simulation RMSE, but the improvements significantly reduce after a certain value (e.g. $\alpha=1.2$ for $y_1$ and $y_3$), making it useless to choose a greater $\alpha$, which will only provide an increase in the one-step-ahead error, as shown from Fig. \ref{f:Tau_var_alpha}.

\begin{figure}[thpb]
	\centering
	\includegraphics[width=1.0\columnwidth]{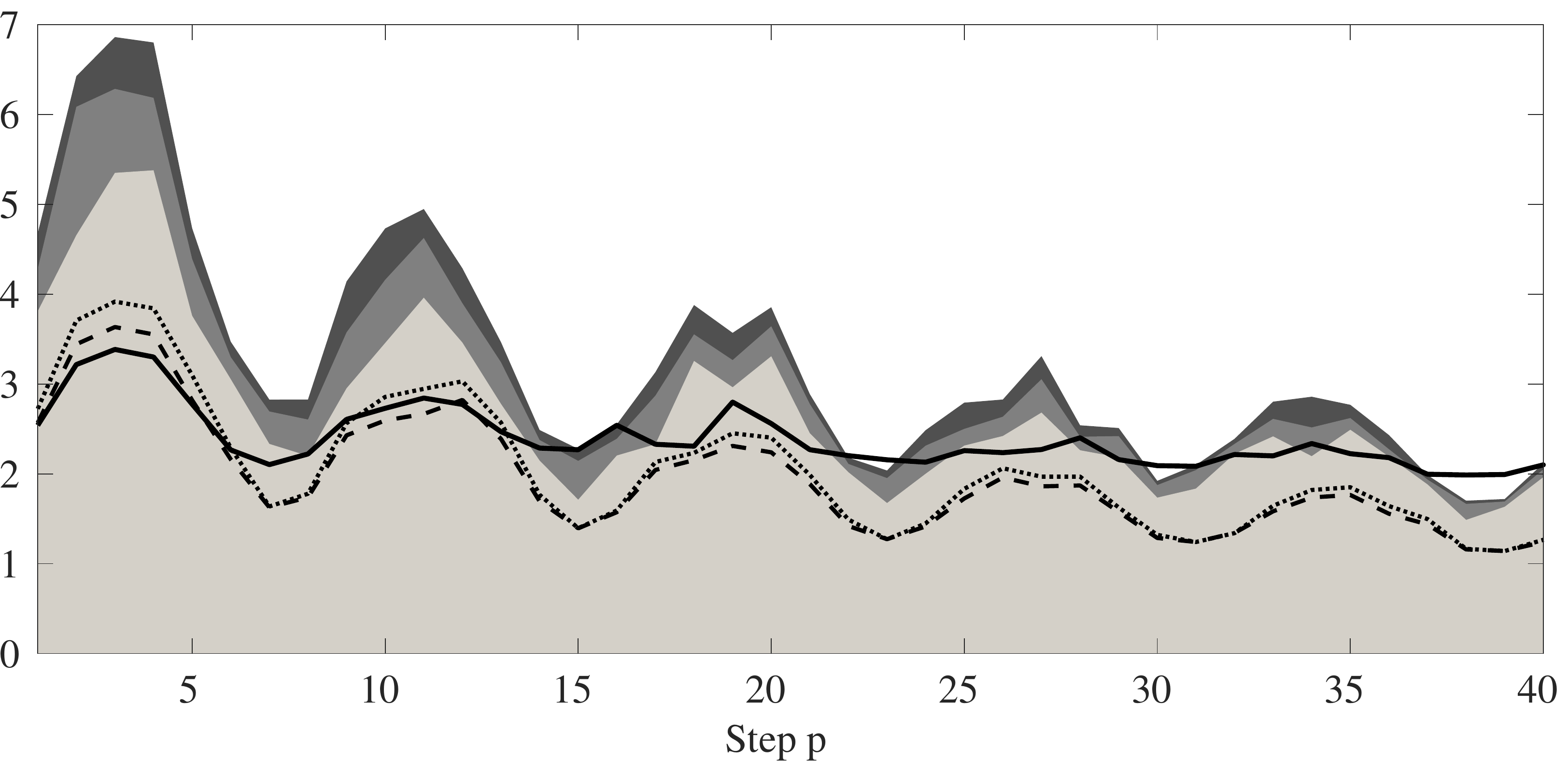}
	\caption{Numerical example: worst-case validation error $e_{2_p}$ and guaranteed simulation error bound $\hat{\tau}_{2_p}$ on $\hat{z}_2$ for the ARX predictor identified using Method II for different values of $\alpha$. Solid line: $e_{2_p}$ for $\alpha=1.0$; dashed line: $e_{2_p}$ for $\alpha=1.1$; dotted line: $e_{2_p}$ for $\alpha=1.2$; light gray area: $\hat{\tau}_{2_p}$ for $\alpha=1.0$; medium gray area: $\hat{\tau}_{2_p}$ for $\alpha=1.1$; dark gray area: $\hat{\tau}_{2_p}$ for $\alpha=1.2$.}
	\label{f:Tau_var_alpha}
\end{figure} 
\begin{table}[ht]
\caption{\label{t:G_ARX_RMSE_var_alpha}Numerical example, validation data: simulation Root Mean Square Error for ARX models for different values of $\alpha$.}
\centering
\setlength\tabcolsep{4.75pt}
\small
\begin{tabular}{c c c c c c c c c}
\toprule
RMSE & $\alpha$: & $1.0$ & $1.05$ & $1.1$ & $1.15$ & $1.2$ & $1.25$ & $1.3$ \\ 
\midrule
& $y_1$ & 2.67 & 1.89 & 1.70 & 1.50 & 1.29 & 1.27 & 1.27 \\
Method II & $y_2$ & 0.70 & 0.57 & 0.57 & 0.57 & 0.57 & 0.57 & 0.57\\
& $y_3$ & 0.14 & 0.08 & 0.07 & 0.07 & 0.06 & 0.06 & 0.06\\
\bottomrule
\end{tabular} 
\normalsize
\end{table}

Finally, Tables \ref{t:G_eig} and \ref{t:G_sys_matrix} report a comparison between the eigenvalues and the $A$ and $B$ matrices of the discrete-time system, obtained applying the trapezoid approximation rule to \eqref{eq:underdamp_sys_ss}, and those of the model identified with the state-space predictors, for the various identification approaches.
\begin{table}[ht]
\caption{\label{t:G_SS_RMSE}Numerical example, validation data: Root Mean Square Error for $p$-step-ahead prediction and simulation for state-space models.}
\centering
\setlength\tabcolsep{3.5pt}
\small
\begin{tabular}{c c c c c c c c}
\toprule
\multicolumn{2}{c}{RMSE} & $p=1$ & $p=10$ & $p=20$ & $p=30$ & $p=60$ & sim \\ 
\midrule
 & $y_1$ & \textbf{1.041} & 1.437 & 1.790 & 1.947 & 2.181 & 2.214 \\
PEM & $y_2$ & \textbf{0.657} & 0.673 & 0.731 & 0.776 & 0.821 & 0.832 \\ 
 & $y_3$ & \textbf{0.067} & 0.067 & 0.070 & 0.073 & 0.078 & 0.078 \\
\midrule
 & $y_1$ & 1.501 & 1.342 & 2.085 & 1.422 & 0.747 & 0.627 \\
SEM & $y_2$ & 0.824 & 0.907 & 0.612 & 0.663 & 0.604 & 0.607 \\ 
 & $y_3$ & 0.073 & 0.099 & 0.079 & 0.079 & 0.075 & 0.075 \\ 
\midrule
 & $y_1$ & 1.067 & 1.110 & 1.272 & 1.184 & 1.242 & 1.260 \\
Method I & $y_2$ & 0.716 & 0.646 & 0.630 & 0.638 & 0.647 & 0.654 \\ 
 & $y_3$ & \textbf{0.067} & \textbf{0.065} & 0.061 & 0.062 & 0.063 & 0.064 \\ 
\midrule
 & $y_1$ & 1.061 & \textbf{1.069} & \textbf{1.026} & \textbf{0.730} & \textbf{0.604} & \textbf{0.584} \\
Method II & $y_2$ & 0.726 & \textbf{0.620} & \textbf{0.600} & \textbf{0.600} & \textbf{0.582} & \textbf{0.584} \\ 
 & $y_3$ & 0.069 & \textbf{0.065} & \textbf{0.060} & \textbf{0.060} & \textbf{0.059} & \textbf{0.059} \\ 
\bottomrule
\end{tabular} 
\normalsize
\end{table}

\begin{table}[ht]
	\caption{\label{t:G_eig}Numerical example: real and identified system eigenvalues.}
	\centering
	\setlength\tabcolsep{3.5pt}
	\small
	\begin{tabular}{c c}
		\toprule
		& Eigenvalues \\ 
		\midrule
		True system (trapezoid approximation)& $0.889\pm i 0.369 \, , \; 0.333$ \\
		\midrule
		PEM (state-space predictor) & $0.877\pm i0.369\, , \; 0.020$ \\
		\midrule
		SEM (state-space predictor) & $0.885\pm i0.372\, , \; 0.723$ \\
		\midrule
		Method I (state-space predictor) & $0.884\pm i0.369\, , \; 0.349$ \\
		\midrule
		Method II (state-space predictor) & $0.885\pm i0.373\, , \; 0.213$ \\
		\bottomrule
	\end{tabular} 
	\normalsize
\end{table}
\begin{table}[ht]
	\caption{\label{t:G_sys_matrix}Numerical example: real and identified system parameters.}
	\centering
	\setlength\tabcolsep{2.8pt}
	\small
	\begin{tabular}{c c c}
		\toprule
		& A & B \\ 
		\midrule
		\begin{tabular}{@{}c@{}} True system (trapezoid\\ approximation) \end{tabular} & $\begin{bmatrix} 0.979 & -0.564 & -9.335 \\ 0.096 & 0.895 & -1.964 \\ 0.004 & 0.058 & 0.265 \end{bmatrix}$ & $\begin{bmatrix} 15.91 \\ 0.785 \\ 0.021 \end{bmatrix}$ \\
		\midrule
		\begin{tabular}{@{}c@{}} SEM \\ (state-space predictor) \end{tabular} & $\begin{bmatrix} 1.095 & -1.882 & 3.252 \\ 0.090 & 0.976 & -2.819 \\ 0.006 & 0.039 & 0.422 \end{bmatrix}$ & $\begin{bmatrix} 15.21 \\ 0.817 \\ -0.016 \end{bmatrix}$ \\
		\midrule
		\begin{tabular}{@{}c@{}} Method II \\ (state-space predictor) \end{tabular} & $\begin{bmatrix} 0.963 & -0.448 & -10.38 \\ 0.111 & 0.760 & -0.647 \\ 0.003 & 0.059 & 0.261 \end{bmatrix}$ & $\begin{bmatrix} 16.03 \\ 0.557 \\ 0.031 \end{bmatrix}$ \\
		\bottomrule
	\end{tabular} 
	\normalsize
\end{table}
\subsection{Experimental case study}
\label{s:exp_res}
\begin{figure}[thpb]
	\centering
	\includegraphics[width=1\columnwidth ,trim={0 3cm 0 0},clip]{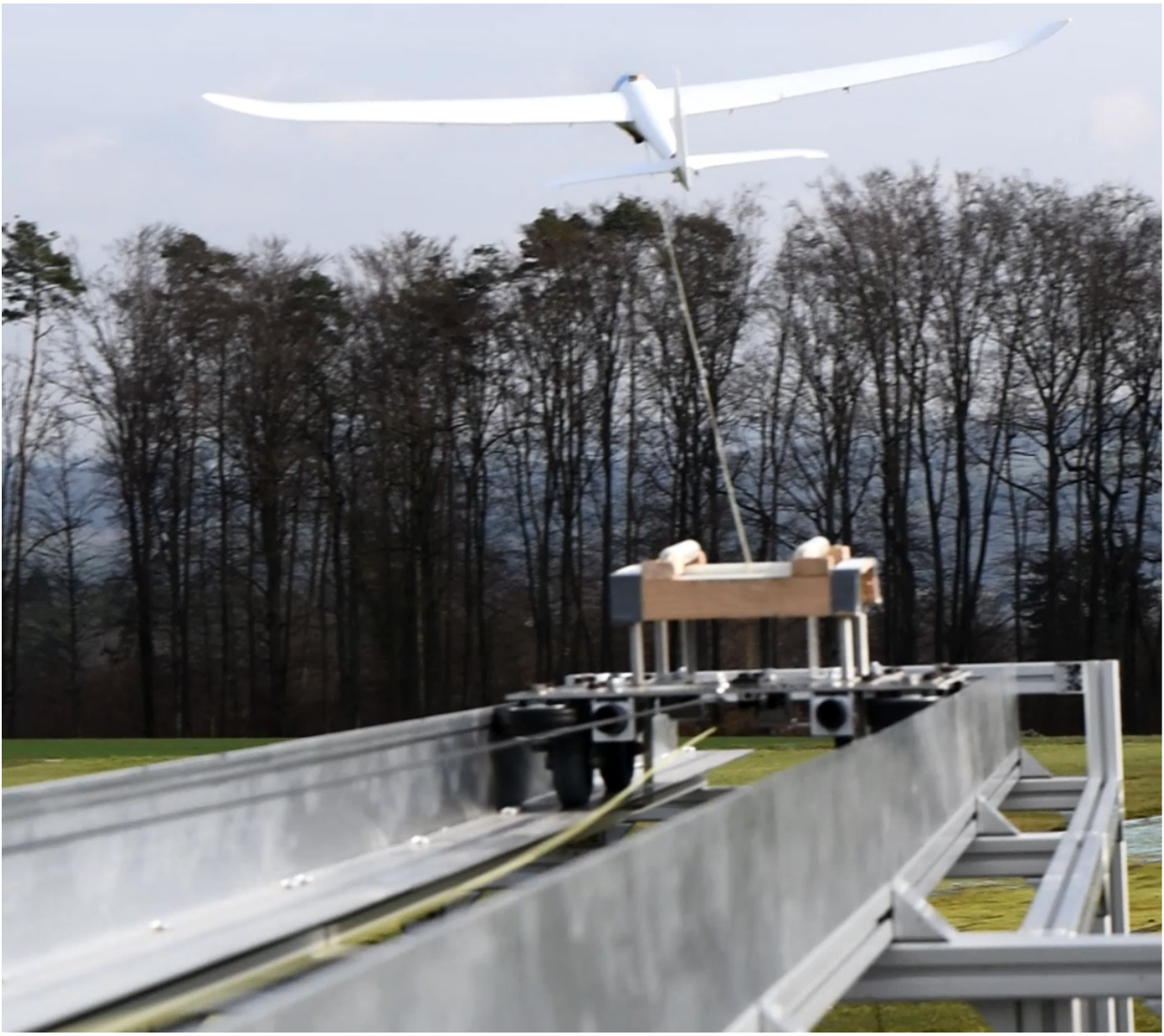}
	\caption{Experimental case study: considered tethered aircraft during an autonomous take-off maneuver.}
	\label{f:glider}
\end{figure} 
Here, we present the results obtained with the proposed identification approach applied to data acquired from real-world test flights of a small-scale prototype of an autonomous tethered aircraft, used for Airborne Wind Energy (AWE) generation, see Fig. \ref{f:glider} and \cite{fagiano2018glider}. We focus on the identification of a model of the roll-rate dynamics of the aircraft, resorting to a data set collected during several experiments. The data acquisition begins right after the take-off phase of each test flight, when the aircraft starts performing  eight-shaped flight patterns parallel to the ground. The system description, along with more detail about the measurements and data set acquisition, is available in \cite{fagiano2018glider}.

As a first approximation, the dynamical equation for the roll angle of the aircraft is given by:
\begin{equation}
\label{eq:glider_roll_eq}
\ddot{\sigma}(t)=a_{\sigma}\dot{\sigma}(t)+b_{\sigma}u(t),
\end{equation}
where $a_{\sigma}$ and $b_{\sigma}$ are parameters to be identified, and $u(t)$ is the control input for the ailerons. Equation \eqref{eq:glider_roll_eq} is a reasonable linear approximation of the nonlinear turning dynamics when the aircraft flies parallel to the ground, as in the considered experiments. The aircraft is autonomous, i.e. it features a feedback controller that manipulates the aileron, rudder, and front propeller to achieve the desired figure-of-eight patterns, which are typical of AWE applications. The data set includes measures of the roll rate and of the ailerons input signal, acquired with a sampling frequency of 50 Hz, given by $\tilde{y}(t)=x(t)+d(t)$, where $d(t)$ is the unknown measurement noise, and $\tilde{u}(t)$, respectively. The identification data set is composed of 11000 samples of each signal, while the validation data set features 6600 data points. 

Since the system state is measurable, we resort to a state-space form predictor of order 1. We apply Procedures \ref{p:d_bar_est_procedure} and \ref{p:decay_est_procedure}, obtaining $\bar{d}=0.82$, $\hat{L}=1.31$, $\hat{\rho}=0.995$ and $\bar{p}=691$. Fig. \ref{f:lambda_glider_SS} depicts the behavior of the error bound $\underline{\lambda}_p$ after the estimation of the disturbance bound $\bar{d}$. Then, we resort to Method II to identify the unknown parameters of \eqref{eq:glider_roll_eq}, obtaining $\hat{a}_{\sigma}=0.959$ and $\hat{b}_{\sigma}=0.120$, and we test the predictor performance against PEM and SEM approaches. Figs. \ref{f:tau_glider_SS} and \ref{f:err_glider_SS} show a performance comparison  in terms of guaranteed simulation error bound $\hat{\tau}_p$ and validation error $e_p$, while Fig. \ref{f:Y_sim_glider_SS} presents an example of time-course of the roll rate, both measured (validation data) and simulated. Table \ref{t:glider_SS_RMSE} shows the RMSE for different horizon lengths. These results confirm that the predictor identified with Method II represents a good trade-off between the PEM and the SEM approaches, combining the one-step-ahead accuracy of the first, with the simulation accuracy over longer horizons of the latter.
\begin{figure}[thpb]
	\centering
	\includegraphics[width=0.99\columnwidth]{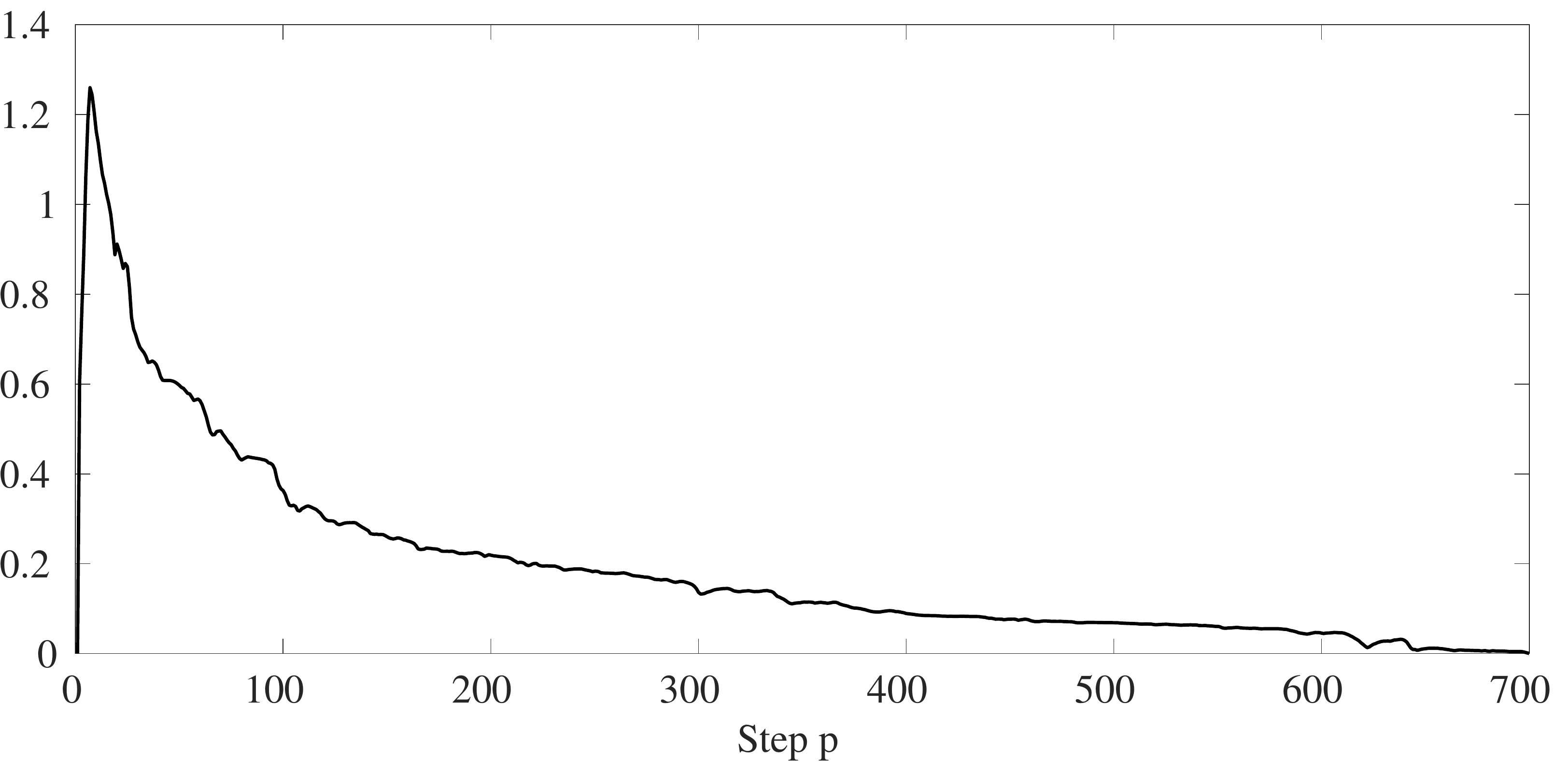}
	\caption{Experimental case study: estimated value of $\underline{\lambda}_p$ using a predictor in the state-space form.}
	\label{f:lambda_glider_SS}
\end{figure}
\begin{figure}[thpb]
	\centering
	\includegraphics[width=0.99\columnwidth]{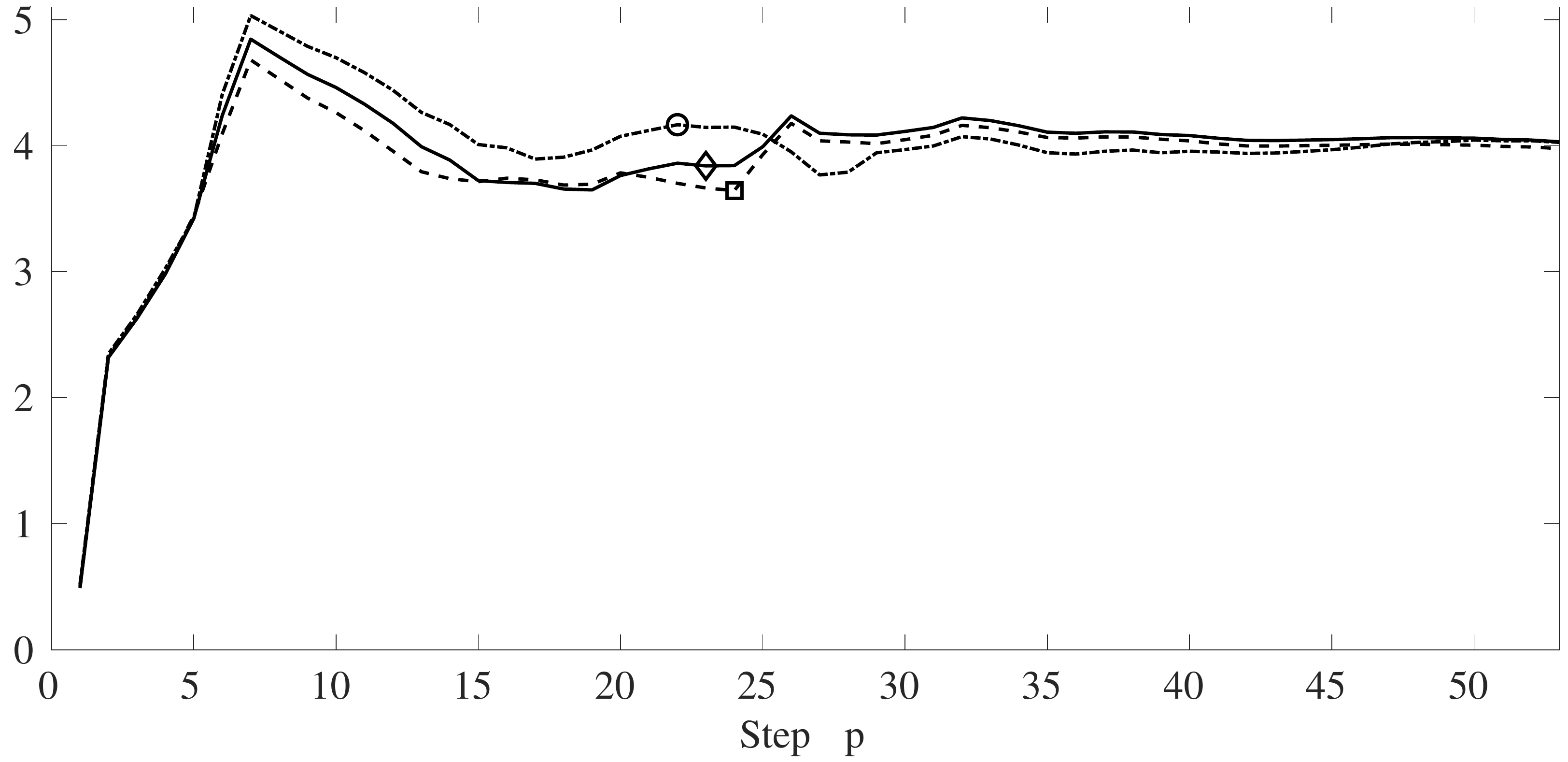}
	\caption{Experimental case study: guaranteed simulation error $\hat{\tau}_p$. Solid line with $\diamond$: Method II; dashed line with $\square$: SEM approach; dash-dotted line with $\circ$: PEM approach.}
	\label{f:tau_glider_SS}
\end{figure}
\begin{figure}[thpb]
	\centering
	\includegraphics[width=0.99\columnwidth]{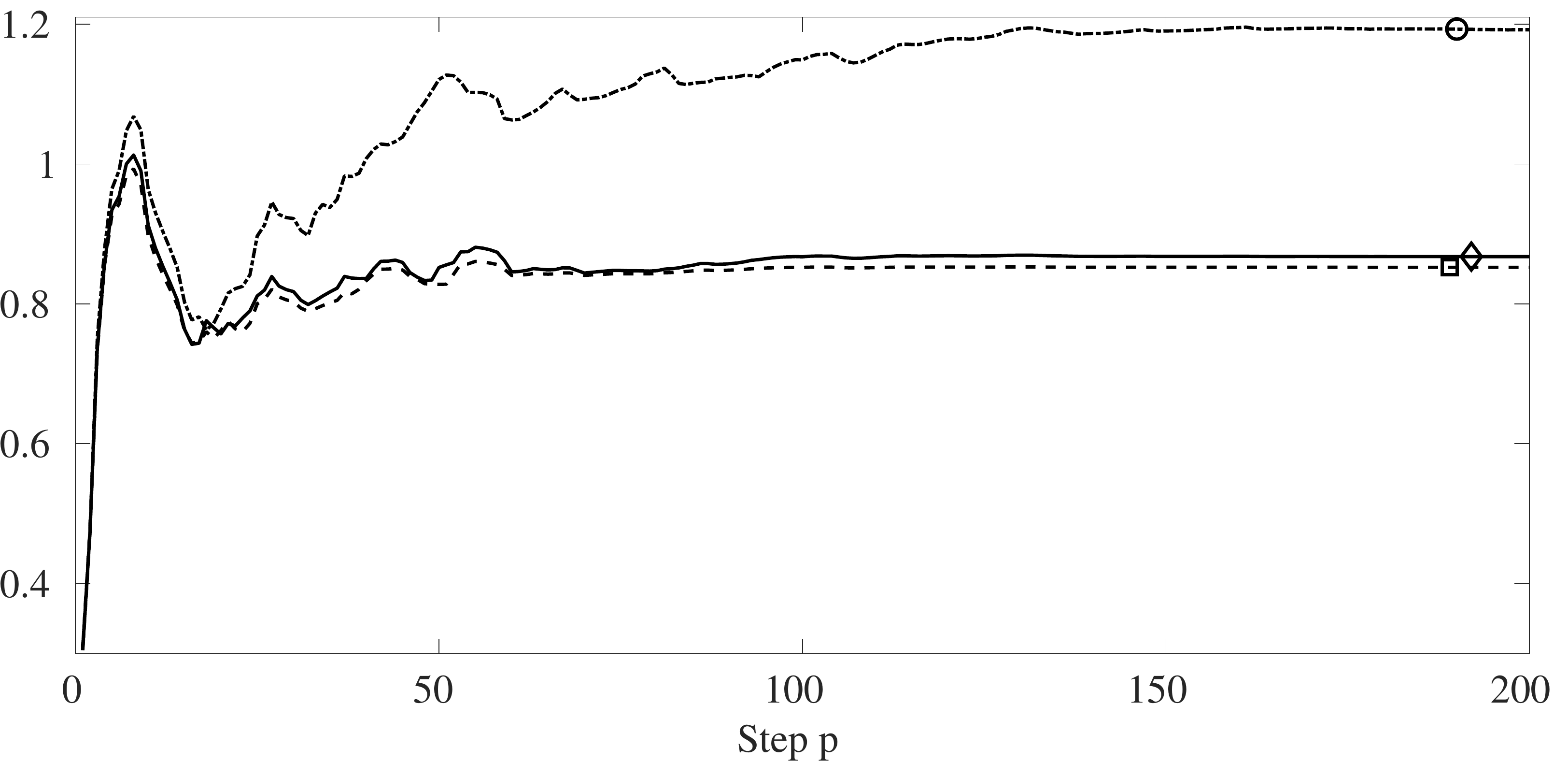}
	\caption{Experimental case study: validation error $e_p$. Solid line with $\diamond$: Method II; dashed line with $\square$: SEM approach; dash-dotted line with $\circ$: PEM approach.}
	\label{f:err_glider_SS}
\end{figure}
\begin{table}[ht]
\caption{\label{t:glider_SS_RMSE}Experimental case study: Root Mean Square Error.}
\centering
\setlength\tabcolsep{3pt}
\small
\begin{tabular}{c c c c c c c c}
\toprule
RMSE & \footnotesize{$p=1$} & \footnotesize{$p=2$} & \footnotesize{$p=10$} & \footnotesize{$p=20$} & \footnotesize{$p=30$} & \footnotesize{$p=60$} & \small sim \\ 
\midrule
PEM & 0.0477 & 0.0811 & 0.196 & 0.251 & 0.264 & 0.293 & 0.324 \\ 
\midrule
SEM & 0.0481 & 0.0813 & 0.188 & 0.227 & 0.230 & 0.231 & 0.232 \\ 
\midrule
Method II & 0.0478 & 0.0810 & 0.189 & 0.231 & 0.234 & 0.232 & 0.234 \\ 
\bottomrule
\end{tabular} 
\normalsize
\end{table}
\begin{figure}[thpb]
	\centering
	\includegraphics[width=0.99\columnwidth]{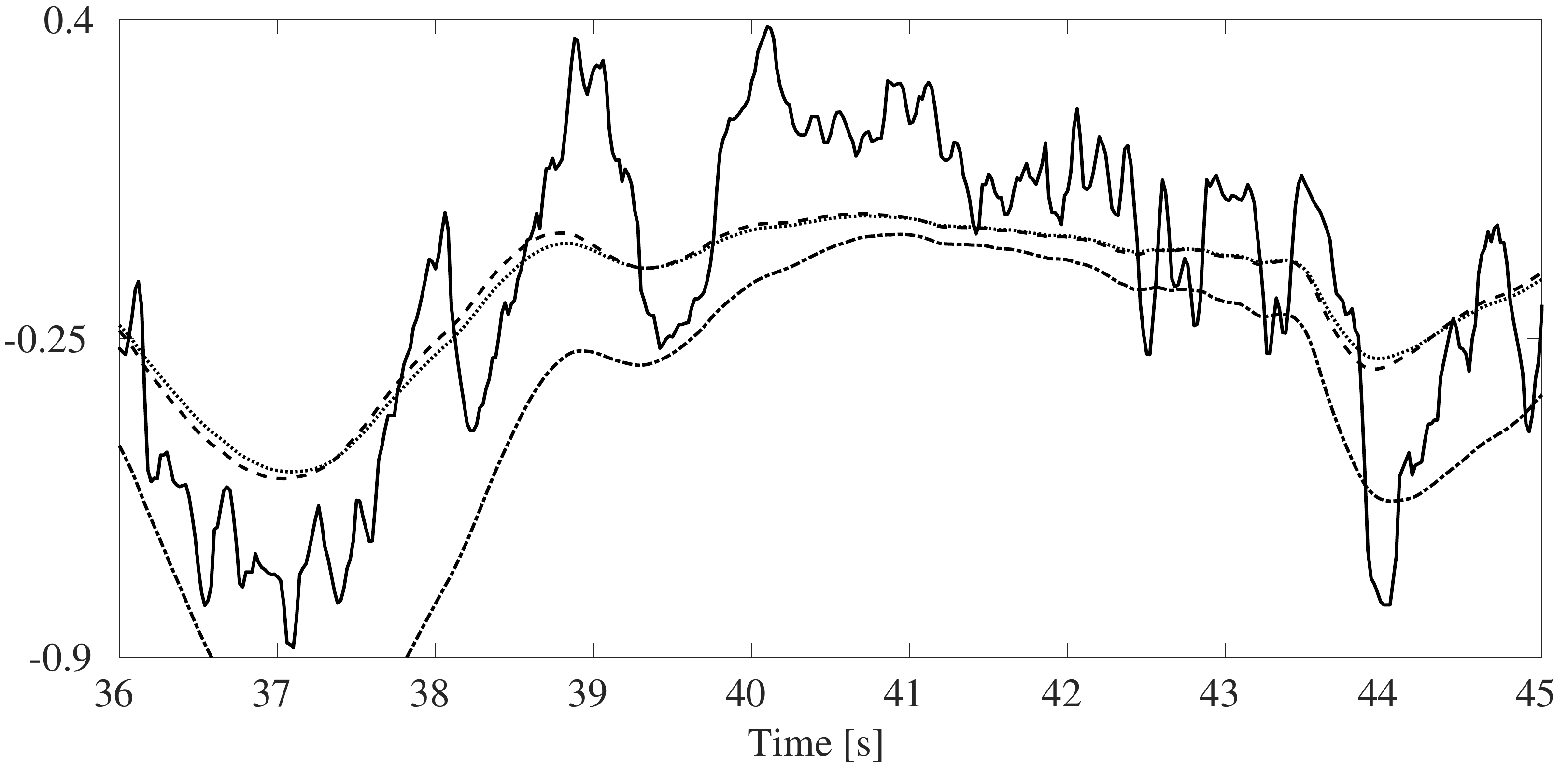}%Y2_sim_glider_SS_zoom2}
	\caption{Experimental case study: simulated roll rate $[\nicefrac{rad}{s}]$ with the state-space predictor. Solid line: measured roll rate $\tilde{y}$; dashed line: simulated roll rate with SEM predictor; dash-dotted line: simulated roll rate with PEM predictor; dotted line: simulated roll rate with Method II predictor.}
	\label{f:Y_sim_glider_SS}
\end{figure}
\section{Conclusions}
\label{s:conclusions}
We presented new results pertaining to the  identification of linear systems with guaranteed simulation error bounds, resorting to a Set Membership framework. The theoretical findings lead to clear procedures to estimate the noise bound, model order and system decay trend. Moreover, we  derived a simulation error bound for an infinite simulation horizon, together with its properties and convergence conditions. This bound allowed us to demonstrate that it is possible to use the decay rate constraints to enforce the asymptotic stability of the identified model. Then, we presented two methods to learn one-step-ahead prediction models exploiting the estimated quantities.  Numerical simulations illustrate the validity and the performance of the proposed identification methods, which we compared to standard PEM and SEM identification approaches. Furthermore, an experimental case study illustrates the applicability on real data. Future work will be devoted to the extension of the proposed identification framework to the nonlinear case.
\appendix
\small
\subsection*{Proof of Theorem \ref{th:conv_lambda_diff_d}}
\noindent From \eqref{eq:disturbed_output}, \eqref{eq:output_gen_pred_form}, and \eqref{eq:phi_set_arx_case}, it follows that:
\[ y_{i_p}=\psi_{i_p}^T\theta_{i_p}^0+d_i=\left(\varphi_{i_p}-\Delta_{i_p}\right)^T\theta_{i_p}^0+d_i, \]
where $\Delta_{i_p} \in \mathbb{D}_{i_p}$, defined in \eqref{eq:D_arx_def}. Thus, \eqref{eq:opt_multistep} can be written as:
\begin{equation}
\label{eq:lambda_in_proof}
\bar{\varepsilon}_{i_p}^0 = \min_{\theta_{i_p} \in \Omega} \max_{\left[\begin{smallmatrix} \varphi_{i_p} \\ y_{i_p} \end{smallmatrix}\right]\in \mathscr{V}_{i_p}} \left( \left\vert \varphi_{i_p}^T \left( \theta_{i_p}^0-\theta_{i_p} \right)-\Delta_{i_p}^T\theta_{i_p}^0 + d_i \right\vert -\bar{d}_i \right).
\end{equation}
Since $\bar{d}_i$ is a constant, we have:
\[ \begin{aligned}
&\max_{\left[\begin{smallmatrix} \varphi_{i_p} \\ y_{i_p} \end{smallmatrix}\right]\in \mathscr{V}_{i_p}} \left( \left\vert \varphi_{i_p}^T(\theta_{i_p}^0-\theta_{i_p})-\Delta_{i_p}^T\theta_{i_p}^0 + d_i \right\vert -\bar{d}_i \right) \\
&\quad = \max_{\left[\begin{smallmatrix} \varphi_{i_p} \\ y_{i_p} \end{smallmatrix}\right]\in \mathscr{V}_{i_p}}  \left\vert \varphi_{i_p}^T(\theta_{i_p}^0-\theta_{i_p})-\Delta_{i_p}^T\theta_{i_p}^0 + d_i \right\vert -\bar{d}_i
\end{aligned} \]
Then, by defining $\Sigma_{i_p} = \varphi_{i_p}^T(\theta_{i_p}^0-\theta_{i_p})-\Delta_{i_p}^T\theta_{i_p}^0 + d_i$, we can write:
\[
\bar{\varepsilon}_{i_p}^0= \min_{\theta_{i_p} \in \Omega^{\phantom{0}}} \begin{cases} \max_{\left[\begin{smallmatrix} \varphi_{i_p} \\ y_{i_p} \end{smallmatrix}\right]\in \mathscr{V}_{i_p}} (\Sigma_{i_p})-\bar{d}_i, \; \text{if} \; \Sigma_{i_p} \geq 0 \\ 
- \min_{\left[\begin{smallmatrix} \varphi_{i_p} \\ y_{i_p} \end{smallmatrix}\right]\in \mathscr{V}_{i_p}} (\Sigma_{i_p})-\bar{d}_i, \; \text{if} \; \Sigma_{i_p} <0
\end{cases}
\]
By definition, the set $\mathscr{V}_{i_p}$  always contains at least an occurrence of $\underline{\varphi}_{i_p}$ and $\underline{y}_{i_p}$ such that:
\[
\bar{\varepsilon}_{i_p}^0= \min_{\theta_{i_p}\in \Omega} \left\vert \underline{\varphi}_{i_p}^T(\theta_{i_p}^0-\theta_{i_p}) - \underline{\Delta}_{i_p}^T\theta_{i_p}^0 +\bar{d}_{0_i} \right\vert - \bar{d}_i
\]
where $\left\vert \underline{\Delta}_{i_p} \right\vert= \left[\bar{d}_{0_i}, \, \hdots, \, \bar{d}_{0_i}, \, 0, \, \hdots, \, 0 \right]^T$. Then, we have that:
\[
\begin{cases} 
\max\limits_{\left[\begin{smallmatrix} \varphi_{i_p} \\ y_{i_p} \end{smallmatrix}\right]\in \mathscr{V}_{i_p}} (\Sigma_{i_p}) = \underline{\varphi}_{i_p}^T(\theta_{i_p}^0-\theta_{i_p}) + \left\Vert \theta_{i_{p,z}}^0 \right\Vert_1 \bar{d}_{0_i} + \bar{d}_{0_i}, \\
\qquad \qquad \qquad \; \text{if} \; \Sigma_{i_p} \geq 0 \\
\\
\min\limits_{\left[\begin{smallmatrix} \varphi_{i_p} \\ y_{i_p} \end{smallmatrix}\right]\in \mathscr{V}_{i_p}} (\Sigma_{i_p}) = \underline{\varphi}_{i_p}^T(\theta_{i_p}^0-\theta_{i_p}) - \left\Vert \theta_{i_{p,z}}^0 \right\Vert_1 \bar{d}_{0_i} - \bar{d}_{0_i}, \\
\qquad \qquad \qquad \; \text{if} \; \Sigma_{i_p} < 0
\end{cases}
\]
where $\underline{\varphi}_{i_p}^T(\theta_{i_p}^0-\theta_{i_p}) \geq 0$ if $\Sigma_{i_p} \geq 0$, and $\underline{\varphi}_{i_p}^T(\theta_{i_p}^0-\theta_{i_p})<0$ if $\Sigma_{i_p}<0$. Thus, under Assumption \ref{as:model_order}, the optimal choice of $\theta_{i_p}$ that minimizes the resulting $\bar{\varepsilon}_{i_p}^0$ is such that $\underline{\varphi}_{i_p}^T(\theta_{i_p}^0-\theta_{i_p})=0$. Thus, given $\underline{\varphi}_{i_p}$, $\underline{y}_{i_p}$, $\underline{\Delta}_{i_p}$, and the corresponding optimal choice of $\theta_{i_p}$, we have:
\begin{equation}
\label{eq:lambda_remaining_proof}
\bar{\varepsilon}_{i_p}^0=\left\Vert \theta_{i_{p,z}}^0 \right\Vert_1 \bar{d}_{0_i} + \bar{d}_{0_i} - \bar{d}_i
\end{equation}
The term $\left\Vert \theta_{i_{p,z}}^0 \right\Vert_1 \bar{d}_{0_i}$ represents an upper bound of the free response of the system to an initial condition given by $\underline{\Delta}_{i_p}$. For an asymptotically stable system, $\left\Vert \theta_{i_{p,z}}^0 \right\Vert_1$ goes to zero with a decay rate which is upper bounded by $\rho_i$, see \eqref{eq:sists_dec_bound}, leading to:
\begin{equation}
\label{eq:lambda_to_zero_proof}
\left\Vert \theta_{i_{p,z}}^0 \right\Vert_1 \bar{d}_{0_i} \xrightarrow{p \to \infty} 0.
\end{equation}
Thus, from \eqref{eq:lambda_remaining_proof} and \eqref{eq:lambda_to_zero_proof}, it follows that:
$\bar{\varepsilon}_{i_p}^0 \xrightarrow{p \to \infty} \left( \bar{d}_{0_i} - \bar{d}_i \right)$.
Finally, from \eqref{eq:properties_preliminary_a} and \eqref{eq:properties_preliminary_b}, it follows that:
\begin{equation*} 
\begin{aligned}
&\lim_{p\to \infty}{\underline{\lambda}}_{i_p} \leq \lim_{p\to \infty}{\bar{\varepsilon}}_{i_p}^0 = \left( \bar{d}_{0_i} - \bar{d}_i \right), \\
&\lim_{p\to \infty}{\underline{\lambda}}_{i_p} \geq \lim_{p\to \infty}{\bar{\varepsilon}}_{i_p}^0 - \eta = \left( \bar{d}_{0_i} - \bar{d}_i \right)-\eta
\end{aligned} 
\end{equation*}
which proves the result \eqref{eq:thm1_convergence}. Note that, when $\bar{d}_i=\bar{d}_{0_i}$ and $o<n$, then $\underline{\lambda}_{i_p}$ converges (besides a quantity $\eta$ that can be made arbitrarily small with a larger data set) to a non-zero value as $p \to \infty$, due to model order mismatch. The rationale behind this statement is that, if $o<n$, there is at least a pair $(\varphi_{i_p},y_{i_p})\in \mathscr{V}_{i_p}$ such that it is not possible to find a $\theta_{i_p}$ that is able to give $\varphi_{i_p}^T \left(\theta_{i_p}^0 - \theta_{i_p} \right)=0$. This will introduce an additional non-zero term in \eqref{eq:lambda_remaining_proof}, making $\bar{\varepsilon}_{i_p}^0$ converge to a non-zero value as $p \to \infty$.
\subsection*{Proof of Corollary \ref{c:lambda_rate_conv}}
\noindent From Theorem \ref{th:conv_lambda_diff_d} it follows that, if $\bar{d}_i=\bar{d}_{0_i}$, then $\bar{\varepsilon}_{i_p}^0 \xrightarrow{p \to \infty} 0$ and $\underline{\lambda}_{i_p} \xrightarrow{p \to \infty} 0$. From the proof of Theorem \ref{th:conv_lambda_diff_d}, we have that, if $\bar{d}_i=\bar{d}_{0_i}$, for the values $\underline{\varphi}_{i_p}$ and $\underline{y}_{i_p}$ defined previously, \eqref{eq:lambda_in_proof} corresponds to $\bar{\varepsilon}_{i_p}^0 = \left\Vert \theta_{i_{p,z}}^0 \right\Vert_1 \bar{d}_{0_i}$. From \eqref{eq:FPS_with_decay}, it follows that $\bar{\varepsilon}_{i_p}^0 = \left\Vert \theta_{i_{p,z}}^0 \right\Vert_1 \bar{d}_{0_i} \leq n\bar{d}_{0_i} L_i \rho_i^{p+1}$, 
which, combined with \eqref{eq:properties_preliminary_a}, yields $
\underline{\lambda}_{i_p}\leq \bar{\varepsilon}_{i_p}^0-\eta \leq n \bar{d}_{0_i}L_i \rho_i^{p+1}$.
\subsection*{Derivation of equation \eqref{eq:err_bound_arx_pgrande}}
\noindent Let us denote with $\hat{z}(k+j|k)$ the $j$-steps ahead prediction of $z$ obtained using the measured output up to time $k$. Then, it is possible to employ the $\bar{p}$-steps ahead predictor $\hat{z}_i(k+\bar{p}+j|k+j)$ to obtain the $(k+\bar{p}+j)$-steps ahead prediction of the system output, using only data up to time instant $k$, according to
\begin{equation}
\label{eq:pred_appndx_tau_inf}
\hat{z}_i(k+\bar{p}+j|k)= \begin{bmatrix} \hat{z}_i(k+j|k) \\ \vdots \\ \hat{z}_i(k+j-o+1|k) \\ u(k+\bar{p}+j)^T \\ \vdots \\ u(k+j-o+1)^T \end{bmatrix}^T \theta_{i_{\bar{p}}}.
\end{equation}
From \eqref{eq:sim_bound_z} we have that $
\left\vert z_i(k+\bar{p}) -\hat{z}_i(k+\bar{p}|k) \right\vert \leq \hat{\tau}_{i_{\bar{p}}}(\theta_{i_{\bar{p}}})$. 
Since the regressor in \eqref{eq:pred_appndx_tau_inf} features predicted output values in place of the measured ones, we are introducing an additional prediction error, which can be expressed as:
\begin{equation}
\label{eq:apdx_add_err}
\sum_{m=1}^{min\{j,o\}} | y_i(k+j-m+1) - \hat{z}_i(k+j-m+1|k) | \, \theta_{i_{j-m+1}}.
\end{equation}
Having $h_{p,o}(\theta_{i_1}) \in \Gamma_{i_p}, \; \forall p \in [2,\bar{p}]$, from Assumptions \ref{as:bounded_dist}, and equation \eqref{eq:sim_bound_y}, it follows that $\left\vert \theta_{i_{j-m+1}} \right\vert \leq \hat{L}_i\hat{\rho}_i^{\bar{p}+m}$, and
$\big| y_i(k+j-m+1) - \hat{z}_i(k+j-m+1|k) \big| \leq \hat{\tau}_{i_{j-m+1}}(\theta_{i_{j-m+1}})+\bar{d}_i$.
Thus, \eqref{eq:apdx_add_err} can be upper-bounded by 
\begin{equation}
\label{eq:apdx_add_err_bound}
\sum_{m=1}^{min\{j,o\}} \left( \hat{\tau}_{i_{j-m+1}}(\theta_{i_{j-m+1}})+\bar{d}_i \right)\hat{L}_i\hat{\rho}_i^{\bar{p}+m}.
\end{equation}
Finally, adding \eqref{eq:apdx_add_err_bound} to the $\bar{p}$-steps ahead error bound $\hat{\tau}_{i_{\bar{p}}}(\theta_{i_{\bar{p}}})$ leads to \eqref{eq:err_bound_arx_pgrande}. 
\subsection*{Proof of Theorem \ref{th:conv_tau_inf_ARX}}
\noindent Equation \eqref{eq:tau_arx_pgrande_iter} can be written as
$\hat{\tau}_{i_{\ell\bar{p}+j}}(\theta_{i_{\ell\bar{p}+j}}) \leq \hat{\tau}_{i_{\bar{p}}}(\theta_{i_{\bar{p}}}) \sum_{m=0}^{\ell-1} \chi_{i,\bar{p}}^m + \bar{d}_i \left( \sum_{m=0}^{\ell} \chi_{i,\bar{p}}^m -1 \right) +\tau_{i_{max_{\{ j, \ell o \} }}} \chi_{i,\bar{p}}^\ell$.
The geometric series $\sum\limits_{m=0}^{\ell-1} \chi_{i,\bar{p}}^m$ converges to $\frac{1}{1-\chi_{i,\bar{p}}}$ as $\ell\to \infty$ if $\left\vert \chi_{i,\bar{p}} \right\vert <1$. Moreover, $\tau_{i_{max_{\{ j, \ell o \} }}} \chi_{i,\bar{p}}^\ell \xrightarrow{\ell \to \infty}0$, if $\left\vert \chi_{i,\bar{p}} \right\vert <1$. This leads to \eqref{eq:tau_inf_ARX_def}.
\subsection*{Proof of Lemma \ref{l:overshoot_cond}}
\noindent In \eqref{eq:dist_tau_inf_ARX}, the terms multiplying $\hat{\tau}_{i_{\bar{p}}}$ and $\bar{d}_i$ are truncated geometric series, so that $\hat{\tau}_{i_{\ell\bar{p}+j}}$ converges to $\hat{\tau}_{i_{\infty}}$ from below as $\ell\bar{p}+j \to \infty$. The last term is instead a vanishing element, since $\chi_{i,\bar{p}}^{\ell}$ is, and it converges to zero from above as $\ell\to \infty$. Thus, it is possible that $\hat{\tau}_{i_{\ell\bar{p}+j}} > \hat{\tau}_{i_{\infty}}$ for some $\ell$ and $j$ under which \eqref{eq:dist_tau_inf_ARX} has a negative result. Since $\tau_{i_{max_{\{j,\ell o\}}}} \leq \tau_{i_{max}}$, by imposing condition \eqref{eq:dist_tau_inf_ARX} $>0$, where $\tau_{i_{max_{\{j,\ell o\}}}}$ is replaced by $\tau_{i_{max}}$, one finds \eqref{eq:lemma_1_condition}, which defines the values of $\tau_{i_{max}}$ guaranteeing that $\hat{\tau}_{i_{\ell\bar{p}+j}}(\theta_{i_{\ell\bar{p}+j}}) \leq \hat{\tau}_{i_{\infty}}(\theta_{i_{\bar{p}}}), \; \forall \ell,j$.
%\subsection*{Proof of Lemma \ref{l:end_of_overshoot}}
%\noindent Straightforward consequence of Lemma \ref{l:overshoot_cond}.
\subsection*{Proof of Theorem \ref{th:asymptotic_stability}}
\noindent Under Assumption \ref{as:asympt_stable}, system \eqref{eq:sist_desc} is BIBO stable, meaning that 
\begin{equation}
\label{eq:part1_proof_th3}
\exists M>0: \, \left\vert z_i(k) \right\vert < M, \; \forall k \in \mathbb{Z},
\end{equation}
for any initial conditions and for any bounded input signal. Given an ARX predictor defined by its parameter vector $\theta_{i_1}$, if $h_{p,o}(\theta_{i_1}) \in \Gamma_{i_p}, \; \forall p \in [2,\bar{p}]$ and condition \eqref{eq:thm2_convergence} holds, then its infinite-horizon simulation error is bounded by a finite quantity $\hat{\tau}_{i_{\infty}}$. Moreover, if condition \eqref{eq:thm2_convergence} holds, it is possible to calculate a bound for the simulation error for any horizon length, resorting to \eqref{eq:tau_arx_pgrande_iter}, and said bound is finite. This implies that 
\begin{equation}
\label{eq:part2_proof_th3}
\exists M>0: \, \left\vert z_i(k)-\hat{y}_{i_k}(k|1,\theta_{i_1}) \right\vert < M, \; \forall k \in \mathbb{Z}.
\end{equation}
From \eqref{eq:part1_proof_th3} and \eqref{eq:part2_proof_th3}, it follows that $\exists M>0: \, \left\vert \hat{y}_{i_k}(k|1,\theta_{i_1}) \right\vert < M, \; \forall k \in \mathbb{Z}$, meaning that the predictor $\hat{y}_{i_k}(k|1,\theta_{i_1})$ is BIBO stable as well. Since the considered predictor \eqref{eq:1s_pred_classic_form} is an ARX model, it is completely observable and reachable, therefore \eqref{eq:1s_pred_classic_form} is also asymptotically stable.

\normalsize

\bibliographystyle{IEEEtran}

\vfill

\end{document}